\newif\ifpictures
\picturestrue

\newif\ifcomment
\commenttrue

\documentclass[12pt]{amsart}
\usepackage{algpseudocode}
\usepackage{amssymb,amsmath}
\usepackage{amsopn}
\usepackage{mathabx}
\usepackage{xspace}
\usepackage{hyperref}
\usepackage[dvips]{graphicx}
\usepackage[arrow,matrix,curve]{xy}
\usepackage {color, tikz}
\usepackage{wasysym}
\usepackage{multirow}
\usepackage{enumitem}

\numberwithin{equation}{section}
\newtheorem{thm}{Theorem}
\newtheorem{prop}[thm]{Proposition}
\newtheorem{lemma}[thm]{Lemma}
\newtheorem{corollary}[thm]{Corollary}

\theoremstyle{definition}

\numberwithin{thm}{section}

\newtheorem{definitionx}[thm]{Definition}
\newenvironment{definition}
  {\pushQED{\qed}\definitionx}
  {\popQED\enddefinitionx}

\newtheorem{examplex}[thm]{Example}
\newenvironment{example}
  {\pushQED{\qed}\examplex}
  {\popQED\endexamplex}
  
\newtheorem{remarkx}[thm]{Remark}
\newenvironment{remark}
  {\pushQED{\qed}\remarkx}
  {\popQED\endremarkx}

\newcounter{FNC}[page]
\def\newfootnote#1{{\addtocounter{FNC}{2}$^\fnsymbol{FNC}$%
     \let\thefootnote\relax\footnotetext{$^\fnsymbol{FNC}$#1}}}

\newcommand{\C}{\mathbb{C}}
\newcommand{\N}{\mathbb{N}}

\newcommand{\R}{\mathbb{R}}

\newcommand{\Z}{\mathbb{Z}}

\newcommand{ \<}{\langle}
\renewcommand{\>}{\rangle}

\newcommand\cF{{\ensuremath{\mathcal{F}}}\xspace}

\newcommand\cN{{\ensuremath{\mathcal{N}}}\xspace}

\newcommand\cP{{\ensuremath{\mathcal{P}}}\xspace}

\newcommand\cR{{\ensuremath{\mathcal{R}}}\xspace}
\newcommand\cS{{\ensuremath{\mathcal{S}}}\xspace}

\newcommand{\al}{\boldsymbol{\alpha}}

\newcommand{\bom}{\boldsymbol{\omega}}
\newcommand{\ba}{\boldsymbol{a}}
\newcommand{\bb}{\boldsymbol{b}}
\newcommand{\bc}{\boldsymbol{c}}
\newcommand{\bd}{\boldsymbol{d}}
\newcommand{\bbe}{\boldsymbol{e}}

\newcommand{\br}{\boldsymbol{r}}
\newcommand{\bt}{\boldsymbol{t}}
\newcommand{\bu}{\boldsymbol{u}}
\newcommand{\bv}{\boldsymbol{v}}
\newcommand{\bw}{\boldsymbol{w}}
\newcommand{\x}{\mathbf{x}}

\newcommand{\bz}{\boldsymbol{z}}

\newcommand{\ssp}{R}
\newcommand{\smsp}{S}
\newcommand{\mosp}{M}

\definecolor{NiceBlue}{rgb}{0.2,0.2,0.75}
\newcommand{\struc}[1]{{\color{NiceBlue} #1}}

\DeclareMathOperator{\Aff}{Aff}

\DeclareMathOperator{\sgn}{sgn}
\DeclareMathOperator{\conv}{conv}
\DeclareMathOperator{\supp}{supp}

\DeclareMathOperator{\vol}{vol}

\DeclareMathOperator{\sing}{sing}

\DeclareMathOperator{\rank}{rank}

\newcommand{\Pplus}{\mathcal{P}^+}
\newcommand{\Splus}{\mathcal{S}^+}

\title{The algebraic boundary of the sonc-cone}

\begin{document}

\author{Jens Forsg{\aa}rd}

\address{Jens Forsg{\aa}rd, Universiteit Utrecht, Mathematisch Instituut, Postbus 80010, 
3508 TA Utrecht, The Netherlands\medskip}

\email{j.b.forsgaard@uu.nl}

\author{Timo de Wolff}

\address{Timo de Wolff, Technische Universit\"at Braunschweig, Institut f\"ur Analysis und Algebra, AG Algebra, Universit\"atsplatz 2, 38106 Braunschweig,
 Germany\medskip}

\email{t.de-wolff@tu-braunschweig.de}

\subjclass[2010]{
Primary:
14M25, 
14T05, 
55R80,
26C10. Secondary:
12D10, 
14P10, 
52B20 
}
\keywords{Agiform, convexity, discriminant, fewnomials, nonnegative polynomial, positivity, sparsity, sum of nonnegative circuit polynomial, sum of squares}
 
 \begin{abstract}
In this article, we explore the connections between nonnegativity, the theory of $A$-discriminants, and tropical geometry. For an integral support set $A \subset \Z^n$, we cover the boundary of the 
sonc-cone by semi-algebraic sets that are parametrized by families of tropical hypersurfaces.
As an application, we characterization generic support sets
for which the sonc-cone is equal to the sparse nonnegativity cone,
and we describe a semi-algebraic stratification of the boundary of the sonc-cone in the univariate case.
\end{abstract}

\maketitle

\vspace{-4mm}
\begin{center}
	\textit{Dedicated to Bruce Reznick on the occasion of his 66th birthday.}
\end{center}

\section{Introduction}

Describing the cone $\struc{\cP_{1+n, 2\delta}}$ consisting of all nonnegative $(1+n)$-variate homogeneous polynomials of degree $2\delta$ is a central problem in real algebraic geometry. 
A classical approach is to study the sub-cone $\struc{\Sigma_{1+n,2\delta}}$
consisting of all sums of squares (SOS).
In 1888, Hilbert obtained the seminal result that $\Sigma_{1+n,2\delta} = \cP_{1+n,2\delta}$ if and only if $n=1$, or $\delta = 1$, or $n=\delta=2$ \cite{Hilbert:Seminal}.
That $\Sigma_{1+n,2\delta}$ and $\cP_{1+n,2\delta}$ are distinct in general motivated Hilbert's 17th problem, which was solved in the affirmative by Artin \cite{Artin,Reznick:SurveyHilbert17th}.
Further results describing the relationship between $\Sigma_{1+n,2\delta}$ and $\cP_{1+n, 2\delta}$ was obtained only recently.
In 2006, Blekherman proved that for  $\delta \geq 2$, asymptotically in $n$,
almost no nonnegative form is a sum of squares  \cite{Blekherman:SignificantlyMoreNNThanSOS}.
In the case $(n, \delta) =(2, 3)$ of ternary sextics
the Cayley--Bacharach relation yields an obstruction for positive forms to be SOS \cite{Blekherman:NNandSOS}. Similar results hold in the case $(n, \delta)=(3, 2)$ of quaternary quadrics.

While the boundary of $\cP_{1+n, 2\delta}$ is contained in a discriminantal hypersurface  \cite{Nie:NonnegativityCone}, the structure of the boundary of $\Sigma_{1+n,2\delta}$, as
a space stratified in semi-algebraic varieties, is more complicated.
For ternary sextics (and similarly for quaternary quartics), 
Blekherman et al.\ showed that the algebraic boundary of the SOS cone has a unique non-discriminantal irreducible component of degree $83200$, which is the Zariski closure of the locus of all sextics that are sums of three squares of cubics \cite{Blekherman:Hauenstein:Ottem:Ranestad:Sturmfels}.
Despite these developments, no complete algebraic description of the SOS cone is known.

\smallskip
In this article, we take the sparse approach to nonnegativity, where one replaces the degree and arity bounds by an explicit set of monomials. For technical reasons, we prefer exponential
sums to polynomials. Thus, we consider a \struc{\emph{support set}} $\struc{A}  = \{\al_0, \dots, \al_d\} \subset \R^n$, and the 
real vector space $\struc{\R^A}$ consisting of all exponential sums
\[
f(\x) = \sum_{i = 0}^d a_i \, e^{\<\al_i, \x\>}
\]
\emph{\struc{supported}} on $A$, where $a_i \in \R$.
We are interested in the \struc{\emph{sparse nonnegativity cone}}
\[
\struc{\Pplus_A} = \big\{ \,f \in \R^A \, \big| \, f(\x) \geq 0 \, \text{ for all } \x \in \R^n \, \,\big\}.
\]
In this setting, the relevant nonnegativity certificate is existence of
a \emph{sonc-de\-com\-position}, where ``sonc'' stands for ``sum of nonnegative circuit polynomial.''
Let us begin by introducing
\emph{circuit polynomials} and \emph{agiforms}.

\smallskip
Let $\struc{\cN(A)} = \conv(A) \subset \R^n$ denote the \struc{\emph{Newton polytope}}
of $A$. If it is clear from context which support set
 is considered, then we denote the Newton polytope by $\cN$.
We abuse notation and identify $A$ with the matrix whose columns are 
the elements of $A$, with an added top row consisting of the all ones vector.
In block form,
\begin{equation}
\label{eqn:SetA}
A  = \left[\begin{array}{ccc}  1 & \cdots & 1 \\ \al_0 & \cdots & \al_d \end{array}\right].
\end{equation}
A \struc{\emph{circuit}} $\struc{c} \subset A$ is a minimal dependent
subset of $A$.
We say that a circuit $c$ is \struc{\emph{simplicial}} if its convex hull is a 
simplex.
Each circuit $c \subset A$ corresponds, up to a scaling factor,
to a vector $\struc{\bc} \in \ker(A)$ with minimal support.
We have that $c$ is simplicial if and only if the vector $\bc$ can be chosen 
with exactly one negative coordinate. We denote by \struc{$C$} the set
of all simplicial circuits contained in $A$.

\smallskip
Let $c \subset A$ be a simplicial circuit.
Then, it follows from the arithmetic-geometric-mean inequality (AGI)
that the exponential sum
\begin{equation}
\label{eqn:agiform}
 g_c(\x) \ = \ \sum_{i = 0}^d \bc_i \, e^{\<\al_i, \x\>}
\end{equation}
is  nonnegative. Following Reznick \cite{Rez89}, we call \eqref{eqn:agiform}
a \struc{\emph{(simplicial) agiform}}.
Reznick initiated the systematic study of the sub-cone of $\Pplus_A$
generated by all simplicial agiforms.\footnote{Reznick considered the more general
notion of an \emph{agiform}, and proved that the edge generators of the cone of all
agiforms are simplicial agiforms. All agiforms appearing in this paper are simplicial.}
The \emph{equality}-part of the AGI implies that $g_c(\x)$ vanishes on the linear
space
\begin{equation}
\label{eqn:AgiformSingular}
\sing(g_c) = \Aff(c)^\perp,
\end{equation}
where $\struc{\Aff}$ stands for the affine hull. Here, $\struc{\sing}$ stands for ``singular'';
since $g_c(\x)$ is nonnegative, any point of vanishing in $\R^n$ will be a singular point.
We will justify using the phrase \emph{singular locus}, rather than \emph{vanishing locus},
in a moment, when describing the relationship with discriminants.

Simplicial agiforms were generalized to \textit{nonnegative circuit polynomials} by Iliman and the second author \cite{Iliman:deWolff:Circuits}; there are three important differences.
First, if $\bw \in \R^n$, then the composition $g_c(\x - \bw)$ is formally not an agiform in Reznick's sense, but it is a nonnegative circuit polynomial.
Second, exponential monomials with a nonnegative coefficient
are  considered  to be circuit polynomials.
Third, circuit polynomials, in difference to Reznick's agiforms, are not necessarily singular.
From a modern perspective, the relationship between agiforms and circuit polynomials can be summarized in that the singular circuit polynomials are obtained from simplicial agiforms
through the group action
\begin{equation}
\label{eqn:agiform}
(\bw, g_c) \mapsto g_c(\x - \bw).
\end{equation}
Notice that the singular locus of $h(\x) = g_c(\x - \bw)$ is the affine space $\bw + \Aff(c)^\perp$
and, hence, we can think of $\bw$ as acting by translating the singular loci of $g_c$.

\smallskip
The \struc{\emph{sonc-cone}} $\struc{\Splus_A} \subset \Pplus_A$ is the cone generated by all
nonnegative circuit polynomials in $\R^A$,\footnote{We use the symbol $\subset$ as meaning ``subset or equal to.''}. The sonc-cone was later studied independently by Chandrasekaran and 
Shah \cite{Chandrasekaran:Shah:SAGE} under the name \textit{sage}.
The three overlapping naming conventions are unfortunate. 
We use the term ``agiform'' throughout, to emphasize the importance of the real singular locus, and as an homage to Reznick's pathbreaking work \cite{Rez89} from three decades ago.
A slight warning is appropriate, as we include also \eqref{eqn:agiform} in the class
of agiforms. 

\smallskip
To state our main result, we need to introduce the notion of a $\Lambda$-discriminant.
Here, $\struc{\Lambda}$ denotes a regular subdivision of the Newton polytope
$\cN$. We write $\Lambda$ as the set, whose elements are the
closed cells of the subdivision. Each $\lambda \in \Lambda$
defines a set of simplicial circuits
\[
\struc{C_\lambda} = \{\,c \in C \, \mid \, c \subset \lambda \,\}.
\]
The subset of $A$ \struc{\emph{covered by simplicial circuits in $\Lambda$}} is
the set
\[
\struc{A_\Lambda} = \bigcup_{\lambda \in \Lambda} \bigcup_{c \in C_\lambda} c.
\]

\begin{definition}
\label{def:LambdaDiscriminant}
Let $A$ be a real configuration,
and let $\Lambda$ be a regular subdivision of the Newton polytope $\cN$.
The \struc{\emph{$\Lambda$-discriminant}}, denoted \struc{$D_\Lambda$},
is the locus of all exponential sums of the form
\begin{equation}
\label{eqn:LambdaDiscriminant}
f(\x) = \sum_{\lambda \in \Lambda} \sum_{c \in C_\lambda} t_{c} \, g_c(\x - \bw_\lambda)
 + \sum_{\al \notin A_\Lambda} t_{\al} \,e^{\<\x, \al\>},
\end{equation}
where $g_c$ is a simplicial agiform supported in $c$,
where $\bw = \{\bw_\lambda\}_{\lambda \in \Lambda}$ is a family
real parameters such that, for any pair $\lambda_i, \lambda_j \in \Lambda$
whose intersection $\lambda_i \cap \lambda_j$ is nonempty,
\begin{equation}
\label{eqn:Binomials}
\bw_{\lambda_i} - \bw_{\lambda_j} \in \Aff(\lambda_{i} \cap \lambda_{j})^\perp,
\end{equation}
and $\bt = \{t_c\}_{c \in C} \cup \{t_{\al}\}_{\al \in A}$ is a family of real parameters.\footnote{
Since we write $\Lambda$ as a set of closed cells, a lower-dimensional circuit $c$ can
be contained in more than one cell $\lambda \in \Lambda$. However, the conditions
\eqref{eqn:Binomials} ensures that all agiforms supported on $c$ appearing in \eqref{eqn:LambdaDiscriminant} has the same singular locus and, hence, they 
differ only by a scaling factor. That is, there is no loss of generality in assuming that
each simplicial circuit $c \in C$ appears at most once in the first sum of \eqref{eqn:LambdaDiscriminant}.}
\end{definition}

Our main result is the following theorem. 

\begin{thm}
\label{thm:Main}
The boundary of the sonc-cone $\Splus_A$ is contained in the union of the coordinate hyperplanes and the $\Lambda$-discriminants,
as $\Lambda$ ranges over all regular subdivision of the Newton polytope $\cN(A)$.
\end{thm}

\begin{remark}[The connection to tropical geometry]
\label{rem:Tropical}
In Definition \ref{def:LambdaDiscriminant}, we allowed the parameters
range $\bt$ and $\bw$ range over all real values fulfilling \eqref{eqn:Binomials}. Our proof of Theorem~\ref{thm:Main} allows a stronger statement, with further restrictions on the parameters.
See Theorem~\ref{thm:MainStrata},
which we state only in the integral case.
If we require that $\bt \geq 0$ and that the parameters $\bw$ are \emph{arranged according to a tropical
variety dual to $\Lambda$} (see \S \ref{sec:TropicalArrangements} 
and \S\ref{sec:PositiveDiscriminant}),
then we obtain a semi-algebraic
set $S_\Lambda \subset D_\Lambda \cap \Splus_A$.
We prove that the union of the coordinate axes and the semi-algebraic sets $S_\Lambda$
contains the boundary of the sonc-cone.
\end{remark}

As an application, we characterize generic support sets 
$A$ for which the sonc-cone equals the sparse nonnegativity cone,
disproving a conjecture by Chandrasekaran, Murray, and Wierman \cite[Conjecture 22]{Chandrasekaran:Murray:Wierman},
see Theorem~\ref{thm:NonnegativeIsSonc}, and 
Examples~\ref{ex:Equality1} and \ref{ex:lastexample}.

\bigskip
Let us end this introduction by revisiting the integral case, when $A \subset \Z^n$.
Then, the substitution $e^{\x} = \bz$ turns $\R^A$ into 
sparse family of polynomials. This is perhaps the most interesting case, and most of
our examples are presented in this context.
Since the exponential function is a
diffeomorphism $\R^n \rightarrow \R_+^n$, membership in
$\Pplus_A$ corresponds, in the polynomial case, to nonnegativity over the positive orthant.
We have not translated our results to the case of nonnegativity 
of polynomials over the reals, but we note that there are two standard approaches. The first is to certify
nonnegativity separately for each orthant. The second is to either restrict to an orthant
of $\R^A$, or impose conditions on the support set $A$, 
as to ensure that the global minimum is attained in $\R_+^n$.

\smallskip
Let us justify the name \emph{$\Lambda$-discriminant}.
In the integral case, the locus \eqref{eqn:LambdaDiscriminant} is algebraic
in the parameters $\bt$ and $e^{\bw}$. 
In particular, if $\struc{\overline{D}_\Lambda}$ denotes the Zariski closure of
$D_\Lambda$ in the complex affine space $\C^A = \overline{\R^A}$,
then $\overline{D}_\Lambda$ is a complex affine algebraic variety.
If $\Lambda$ is the trivial regular subdivision of $\cN$,
that has a unique top-dimensional cell, then, under mild conditions,\footnote{For example, that at least one point of $A$ is contained in the interior of $\cN$.} $\overline{D}_\Lambda$ 
coincides with the eponymous \emph{$A$-discri\-minant} \cite{GKZ94}.
In fact, \eqref{eqn:LambdaDiscriminant} is nothing but
the Horn--Kapranov uniformization of the $A$-discriminant. 
In analogy with the Horn--Kapranov map, \eqref{eqn:LambdaDiscriminant} yields a
parametrization of the $\Lambda$-discriminant where the parameters represent
scaling factors and locations of singular loci.

\smallskip
Let us stay in the integral case. Restricting the parameters as mentioned
in Remark~\ref{rem:Tropical}, we obtain a covering of the boundary of the
sonc-cone by closed semi-algebraic sets.
By taking a common refinement, we obtain a stratification of the boundary of $\Splus_A$
into semi-algebraic sets. 
We give a complete description of this stratification in the case of a family of univariate exponential sums in \S \ref{sec:Univariate}. 
A number of issues remains to be resolved, before we can
describe this stratification completely in the multivariate case.
The most important issue being to determine the dimension of the intersection of a 
$\Lambda$-discriminant with the boundary of the sonc-cone. Even to describe the dimension
of the $\Lambda$-discriminant itself, in terms of combinatorial properties of $\Lambda$,
which generalizes the recently solved problem of computing the dual-defectivity of 
a toric variety \cite{Est10, Est13, For19}, remains open.

\subsection{Disposition}
We first cover some preliminaries on circuits, \S\ref{sec:circuits}, and on the sonc-cone, \S\ref{sec:SoncCone}.
We introduce sonc-supports and study singular loci of simplicial agiforms in \S\ref{sec:SoncSupports}.
In section \S\ref{sec:Univariate} and \S\ref{sec:TropicalArrangements}, we prove the technical
versions of Theorem~\ref{thm:Main}, first in the univariate case, and then in the multivariate case.
We turn our focus to $\Lambda$-discriminants in \S\ref{sec:PositiveDiscriminant},
completing the proof of our main theorem, and describing the stratification of the boundary of the
sonc cone in the integral univariate case.
We finish with the characterization of generic support sets for which the sonc-cone coincides with the nonnegativity cone in \S\ref{sec:Equality}.

\section*{Acknowledgments}
We thank the anonymous referees for their numerous helpful comments.
This project was initiated during the spring semester on ``Tropical Geometry, Amoebas, and Polytopes'' at Institut Mittag-Leffler (IML), and we thank the Swedish Research Council under grant no.~2016-06596 for the excellent working conditions. 
We cordially thank St{\'e}phane Gaubert for the enlightening discussions.
JF was supported by Vergstiftelsen during his stay at IML.
JF gratefully acknowledges the support of SNSF grant \#159240 ``Topics in tropical and real geometry,'' the NCCR SwissMAP project,
and the Netherlands Organization for Scientific Research (NWO),
grant TOP1EW.15.313.
TdW was supported by the DFG grant WO 2206/1-1.

\section{Preliminaries \emph{on} Circuits}
\label{sec:circuits}

\subsection{Sparse families of exponential sums}
Let $A \subset \R^n$ be a support set, as in the introduction, of cardinality $d$.
The \struc{\emph{dimension}} of $A$, denoted by $\struc{\dim(A)}$, is the dimension of the affine span of the Newton polytope $\cN$. We use the notation $\struc{F \preccurlyeq \cN}$ to denote that
$F$ is a face of the polytope $\cN$.
The \struc{\emph{codimension}} of $A$ is introduced as
\[
\struc{m}= d-\dim(A).
\]
Then, $\rank(A) = 1+\dim(A)$ and, hence, 
$m = \dim (\ker(A))$.

\smallskip
Each element  $\al \in A$ defines a function $\R^n \rightarrow \R$ by the map 
$\x \mapsto e^{\<\x, \al\>}$, and such a function is called an \struc{\emph{exponential monomial}}.
The ordered set $A$ defines the \struc{\emph{exponential toric morphism}} $\varphi_A\colon \R^n \rightarrow \R^{1+d}$
whose coordinates are exponential monomials:
\begin{equation}
\label{eqn:RealToricMorphism}
\struc{\varphi_A(\x)} = \big( e^{\<\x, \al_0\>}, \dots, e^{\<\x, \al_d\>} \big).
\end{equation}
The associated sparse family of exponential sums $\R^A$ is the real vector space 
consisting of all compositions of the exponential toric morphism $\varphi_A$
and a linear form acting on $\R^{1+d}$,
\begin{equation}
\label{eqn:f}
	\struc{\R^A} \ = \ \big\{\,\,f(\x) =  \<\varphi_A(\x), \ba\>\,\, \big| \,\,\ba \in \R^{1+d}\,\,\big\}.
\end{equation}
There is a canonical isomorphism $\R^A \simeq \R^{1+d}$,
and a canonical group action $\R^n\times \R^A \rightarrow \R^A$, given by
$(\bw, f) \mapsto f(\x - \bw)$. Under the identification $\R^A \simeq \R^{1+d}$,
the group action takes the form $(\bw, \ba) \mapsto \ba*\varphi(-\bw)$, where $\struc{*}$ denotes component-wise multiplication.

\begin{example}
\label{ex:Jens}
The matrix
\begin{equation}
\label{eq:AJens}
A = \left[\begin{array}{cccccc} 1 & 1 & 1 & 1 & 1 & 1 \\ 0 & 2 & 3 & 1 & 2 & 1 \\ 0 & 0 & 0 & 1 & 1 & 2 \end{array}\right]
\end{equation}
has $\dim(A) = 2$. The Newton polytope $\cN(A)$ appears in the leftmost picture in Figure~\ref{fig:Circuits}. We have that $d = 6$ and $m = 3$. The family $\R^A$ consists of all bivariate exponential sums
\[
f(x_1, x_2) = a_0 + a_1\, e^{2 x_1} + a_2\, e^{3x_1} + a_3 \,e^{x_1 + x_2}+ a_4 \,e^{2x_1 + x_2}+ a_5 \,e^{x_1 + 2x_2}.\qedhere
\]
\end{example}

\begin{figure}[t]
\includegraphics[height=21mm]{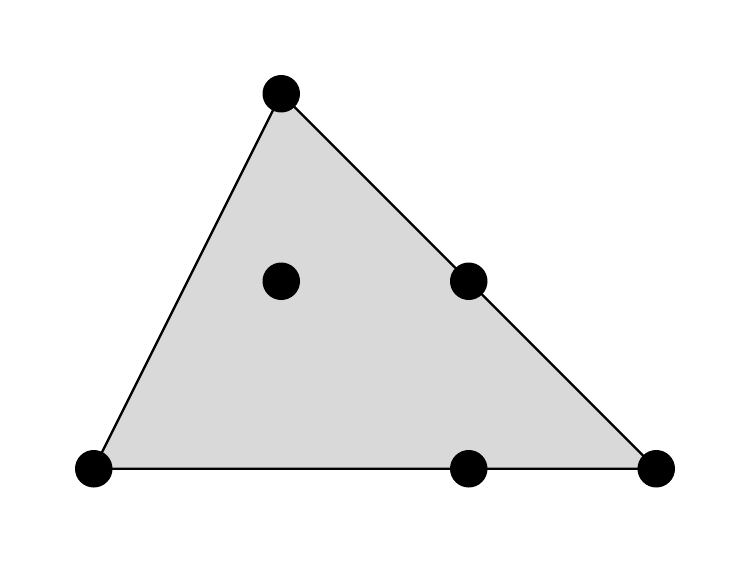}
\hspace{2mm}
\includegraphics[height=21mm]{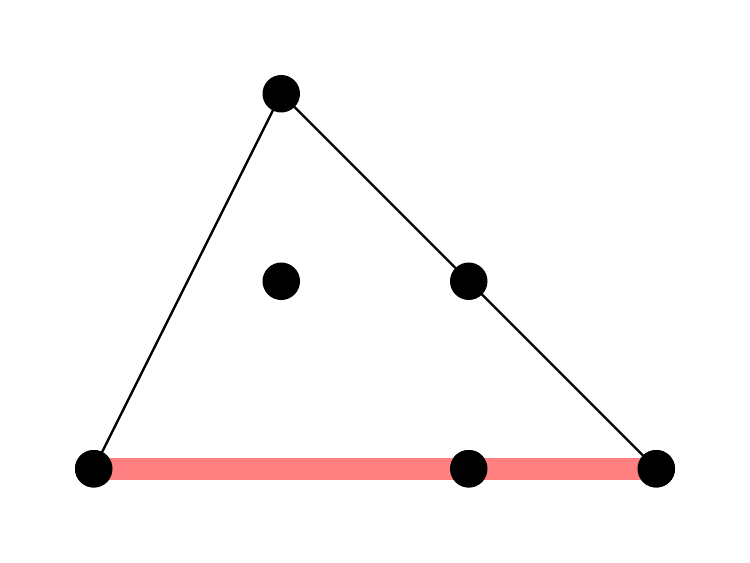}
\hspace{2mm}
\includegraphics[height=21mm]{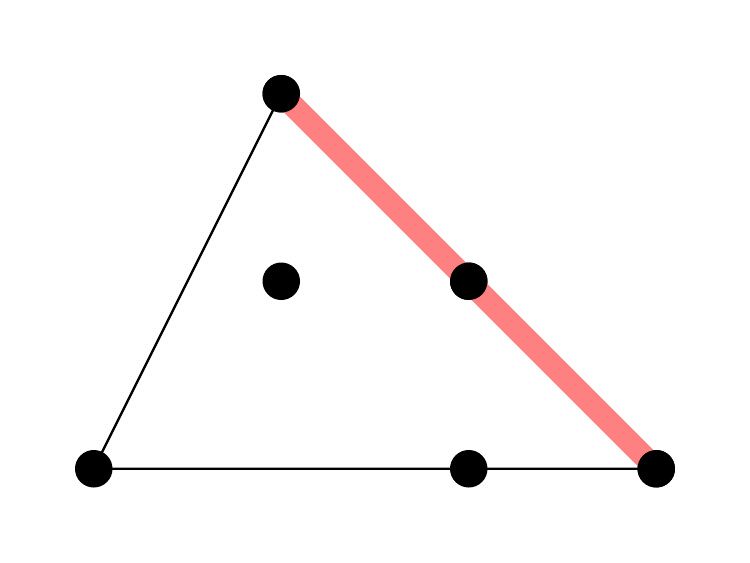}
\hspace{2mm}
\includegraphics[height=21mm]{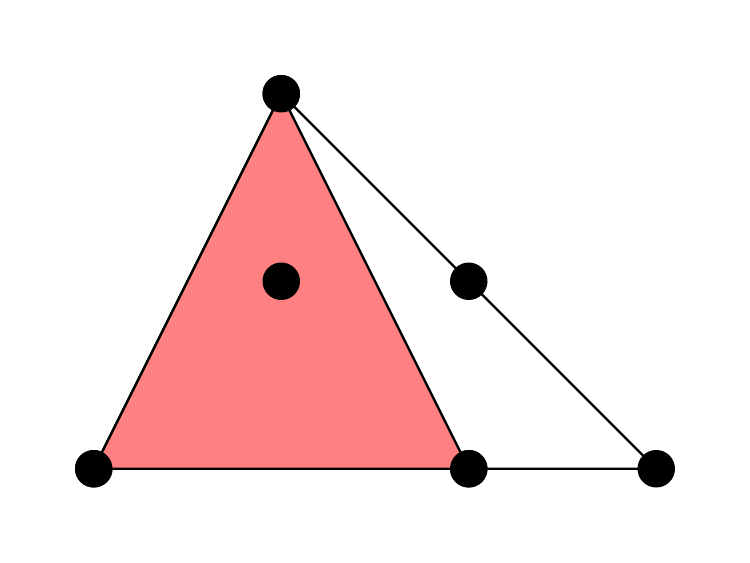}
\hspace{2mm}
\includegraphics[height=21mm]{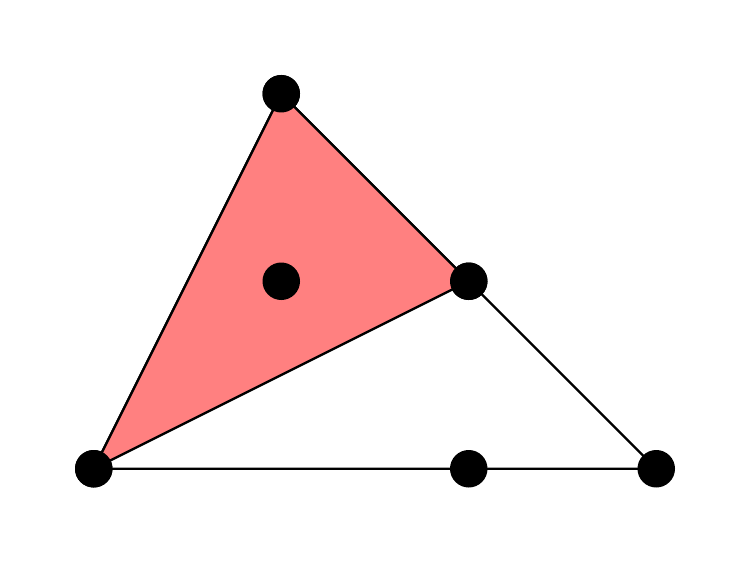}
\caption{The Newton polytope from Example~\ref{ex:Jens},
and the four simplicial circuits which yield edge generators
of $\cR_A$.}
\label{fig:Circuits}
\end{figure}

\subsection{Barycentric coordinates}

\smallskip
Let $c \subset A$ be a simplicial circuit with relative interior point $\al_0$ and vertices $\al_1, \dots, \al_d$. Then, the \struc{\emph{barycentric coordinates}} of $\al_0$ are the convex expression
\begin{equation}
\label{eq:BarycentricCoordinates}
\left\{ \begin{array}{lllllll}
1 & = &\beta_1& + &\dots& + &\beta_d\\
\al_0 & = &\beta_1\, \al_1 &+ &\dots &+ &\beta_d \, \al_d,
\end{array}\right.
\end{equation}
where $\beta_1, \dots, \beta_d > 0$.
It follows that the vector $\bc \in \ker(A)$ is (up to scaling)
\begin{equation}
\label{eqn:CanonicalCircuitVector}
 \bc = (-1, \beta_1, \dots, \beta_d).
\end{equation}
Whenever $c \subset A$ is a simplicial circuit, we call
\eqref{eqn:CanonicalCircuitVector} its \struc{\emph{barycentric coordinate vector}}.
If $A$ is integer, then we can re-scale $\bc$ to obtain a primitive integer vector.
Let us prove a variation of Carath{\'e}odory's Theorem.

\begin{lemma}
\label{lem:SimplexCircuitConstruction}
Let $\al_0\in A$ be a relative interior point of the Newton polytope
$\cN$, and let $\al_1 \in A$
be arbitrary but distinct from $\al_0$. Then, there exists a simplicial
circuit $c \subset A$ with $\al_0$ as its relative interior point
and $\al_1$ as a vertex.
\end{lemma}

\begin{proof}
Let $\ell$ denote the line through $\al_0$ and $\al_1$,
and let $\bv = \al_0 - \al_1$ be a directional vector of $\ell$.
By convexity, $\ell$ intersects the boundary of 
the Newton polytope $\cN$ in two points $\hat \al_0$ and $\hat \al_1$.
The set $\{\hat \al_0, \al_0, \al_1, \hat \al_1\}$ is ordered along $\ell$ by
\[
\<\bv, \hat \al_0\> > \<\bv, \al_0\> > \<\bv, \al_1\> \geq  \<\bv,\hat \al_1\>,
\]
where the first inequality is strict
since $\hat \al_0 \neq \al_0$. In particular
$\{\hat \al_0, \al_0, \al_1\}$ is a one-dimensional (simplicial) circuit.
The barycentric coordinates yields a convex expression
\[
\al_0 = \delta_0 \hat \al_0 + \delta_1 \al_1.
\]
Let $F$ denote the smallest face of $\cN$ containing $\hat \al_0$.
By Carath{\'e}odory's Theorem, there is a convex expression
\[
\hat \al_0 =\beta_2 \, \al_2 + \dots + \beta_r\, \al_r,
\]
where $\al_2, \dots, \al_r \in F \cap A$ forms the vertices of a simplex.
It follows that
\[
\al_0 = \delta_1 \al_1 + \delta_0 \beta_2 \al_r + \dots + \delta_0\beta_r \al_r
\]
is a convex expressions for $\al_0$ in terms of $\{\al_1, \dots, \al_r\}$.
Since $\al_0 \notin F$, we have that $\al_1 \notin F$. Hence, the set
$\{\al_1, \dots, \al_r\}$ forms the vertices of a simplex. It follows that $\{\al_0, \al_1, \dots, \al_r\}$ is a simplicial circuit.
\end{proof}

\subsection{The Reznick cone}
\label{sec:ReznickCone}
Let us revisit the simplicial agiforms appearing in \eqref{eqn:LambdaDiscriminant}, in
the definition of the $\Lambda$-discriminant.
Using \eqref{eqn:RealToricMorphism}, we can write
\[
g_c(\x - \bw_\lambda) = \< \varphi_A(\x - \bw), \bc\>.
\]
In this notation, the first (double) sum of \eqref{eqn:LambdaDiscriminant} can be expressed as
\[
\sum_{\lambda \in \Lambda} \< \varphi_A(\x - \bw), \ba_\lambda\>,
\quad
\text{ where }
\quad
 \ba_\lambda = \sum_{c \in C_\lambda} t_c \, \bc.
\]
If we restrict to $t_c \geq 0$, then 
$\ba_\lambda$ is an element of the cone spanned by the vectors $\bc$.

\begin{definition}
\label{def:ReznickCone}
The \struc{\emph{Reznick cone}} $\struc{\cR_A} \subset \ker(A)$ is the cone spanned by the barycentric coordinate vectors $\bc$ for all simplicial circuits $c \in C$.
We say that $\cR_A$ is \struc{\emph{full-dimensional}} if it has dimension $m = \dim(\ker(A))$.
\end{definition}

Properties of the Reznick cone
are often reflected in properties of the sonc-cone. In this section, 
we describe the edge generators of the Reznick cone,
and provide a sufficient conditions for $\cR_A$ to be a simple polyhedral cone.

\begin{thm}
\label{thm:EdgeGenerators}
A simplicial circuit $\bc \in \ker(A)$ is an edge generator of the Reznick cone $\cR_A$ if and only if the set $c \subset A$ is \struc{minimal} in the sense that
\[
\cN(c) \cap A = c.
\]
\end{thm}

\begin{proof}
Let $A  = \{\al_0, \dots, \al_d\}$, and let $c = \{\al_0, \dots, \al_r\}$ be a simplicial circuit of dimension $r$ with relative interior point $\al_0$. Assume first that there is a point $\al_{r+1}\in \cN(c)\cap A$. 
It suffices to consider the case that $A = c \cup \{\al_{r+1}\}$. 

On the one hand, we have that $\al_{r+1}$ is contained in the relative interior of some face $F \preccurlyeq \cN(c)$ and, hence, it is the relative interior point of some
simplicial circuit $c_1 \subset A$. On the other hand, Lemma~\ref{lem:SimplexCircuitConstruction} gives that $\al_{r+1}$
is a vertex of some simplicial circuit $c_2 \subset A$ which has
$\al_0$ as its relative interior point. Since $\ker (A)$ is two dimensional,
there is a linear relation between the vectors $\bc, \bc_1,$ and $\bc_2$.
By \eqref{eqn:EdgeGenInclusion}, we conclude that 
$\bc$ is a positive combination of $\bc_1$ and $\bc_2$.
Since the vectors $\bc_1$ and $\bc_2$ are linearly independent (as they represent distinct circuits),
this implies that $\bc$ is not an edge generator of $\cR_A$.

To prove the converse, 
we first show that, if there is a positive combination
\begin{equation}
\label{eqn:EdgeGenLinComb}
\bc =  t_1\,\bc_1 + \dots + t_k \bc_k,
\end{equation}
where each $c_i$ is a simplicial circuit distinct from $c$, then we have for all $i$ that
\begin{equation}
\label{eqn:EdgeGenInclusion}
\cN(c_i) \subset \cN(c).
\end{equation}
Indeed, if $\al$ is a vertex of the union $\bigcup_i \cN(c_i)$, then $\al$ is not a relative interior point
of any of the circuits $c_i$. That is, the $\al$-coordinate in each term of the right hand side of \eqref{eqn:EdgeGenLinComb} is nonnegative. Since $\al$ is a vertex of at least one of the
$c_i$, the $\al$-coordinate in the right hand side of \eqref{eqn:EdgeGenLinComb} is positive.
Hence, we have a positive coordinate also in the left hand side, which implies that
$\al$ is a vertex of $\cN(c)$.

Now assume that $\bc$ is not an edge generator of $\cR_A$. Then, there exists
a relation of the form \eqref{eqn:EdgeGenLinComb},
where $c_i \subset \cN(c)$ for each $i$.
Since $c$ is a circuit, it holds that $c_i \subset c$ if and only if
$c_i = c$. But $c_i$ is distinct form $c$ and, hence, there exists an
element $\al \in c_i \setminus c$. Then, $\al \in \cN(c) \cap A$
and $ \al \notin c$.
\end{proof}

When searching for a nonnegativity certificate in this sparse setting,
it suffices to consider simplicial circuits which are edge generators of $\cR_A$. 
in special cases, the number of edge generators can be significantly smaller than the total number of simplicial circuits.

\begin{prop}
\label{pro:DimensionOneSimplePolyhedralCone}
If $A$ has dimension $n=1$,
then $\cR_A$ is a simple polyhedral cone.
\end{prop}

\begin{proof}
There is no loss of generality in assuming that $A = \{\alpha_0, \dots, \alpha_d\} \subset \Z$,
where $\alpha_0 < \dots < \alpha_d$. We conclude from Theorem~\ref{thm:EdgeGenerators}
that the edge generators of $\cR_A$ are the 
simplicial circuits $\{\alpha_{i-1}, \alpha_i, \alpha_{i+1}\}$
for $i = 1, \dots, d-1$.
Since all circuits in $A$ are simplicial, the Reznick cone 
has dimension $m$.
Since $m = d-1$, we conclude that $\cR_A$ is simple.
\end{proof}

The Reznick cone is in general not simple. For example,
the support set $A$ from Example~\ref{ex:Jens} is three-dimensional,
but the Reznick cone $\cR_A$ has four edge generators;
the corresponding simplicial circuits are depicted in Figure~\ref{fig:Circuits}.

\begin{prop}
\label{prop:BasisSimplexCircuits}
If the Newton polytope $\cN(A) = \conv(A)$ has a relative interior point, then $\cR_A$ is full-dimensional.
\end{prop} 

\begin{proof}
We show that the vectors $\bc$ span the $\ker(A)$ over $\R$
by a double induction over the dimension $n$ and the codimension $m$. 
The induction bases consist of the cases $m=1$ for arbitrary $n \geq 1$, and $n=1$ for arbitrary $m\geq 1$. 
If $m=1$, then $A$ is a simplicial circuit, and the kernel of $A$
is one-dimensional. If $n=1$, then all circuits $c \subset A$ are simplicial circuits,
and the claim follows from that the set of all circuits span $\ker(A)$.

Let $\al_0$ denote an interior point of $\cN(A)$. Expressing $\al_0$ as a generic convex combination of $\al_1, \dots, \al_d$, yields a vector $\bb_1 \in \ker(A)$
with the sign pattern
\begin{equation}
\label{eqn:SignPattern}
\sgn \bb_1 = (-, + ,+, \dots,+,  +).
\end{equation}
Extend $\bb_1$ to a basis $\{\bb_1, \dots, \bb_m\}$ of $\ker(A)$.  After adding sufficiently large multiple of $\bb_1$ to $\bb_k$, there is no loss of generality in assume that each $\bb_k$ has the sign pattern \eqref{eqn:SignPattern}.

We modify the vectors $\bb_1, \dots, \bb_m$ one by one by elimination. At step $k$,
choose some $j \neq k$, and replace $\bb_k$ with the vector
\[
\hat \bb_k = \bb_k - \varepsilon \bb_j,
\]
where $\varepsilon$ is chosen as the smallest positive number such that $\hat \bb_k$ has at least one vanishing coordinate. The vector $\hat \bb_k$ is nontrivial, since $\bb_k$ and $\bb_j$ are linearly independent. Hence, the first coordinate of  $\hat \bb_k$ must be negative, since $\ker(A)$ does not contain any nontrivial nonnegative vectors.

After iteratively replacing $\bb_k$ by $\hat \bb_k$, we obtain a basis $\{\bb_1, \dots, \bb_m\}$ of $\ker(A)$
fulfilling the following two conditions. First, as $\bb_k$ has a vanishing coordinate,
its support is a strict subset $A_k \subset A$. Second, as the only negative coordinate
of $\bb_k$ is its first coordinate, the support set $A_k$ has $\al_0$ as a relative interior point.
Since a strict subset of $A$ has either strictly smaller dimension or strictly smaller codimension, the induction hypothesis implies that each 
$\bb_k$ can be written as a linear combination of vectors $\bc$ for simplicial circuits $c \subset A_k$. But any simplicial circuit contained in $A_k$ is contained in $A$.
Since $\{\bb_1, \dots, \bb_k\}$ is a basis for $\ker(A)$, the statement follows.
\end{proof}

\begin{example}
The simplest example of a support set $A$ such that $\ker(A)$ is not spanned by simplicial circuits
is the case when $A$ is a non-simplicial circuit. In this case, $\ker(A)$ has dimension
one, but $A$ contains no simplicial circuits. 
\end{example}

\begin{example}
If $A$ is a two-dimensional support set with no relative interior point, 
then all elements $\al \in A$ are located on (one-dimensional) edges of the Newton polytope
$\cN$. One can prove, using Theorem~\ref{thm:EdgeGenerators} and Proposition~\ref{pro:DimensionOneSimplePolyhedralCone}, that, in this case,
the dimension of the Reznick cone is equal to the number of points of $\al$ which 
are not vertices of $\cN$. For example, $\cR_A$ is full-dimensional if and only if $\cN$ is a simplex.
\end{example}

\section{Preliminaries \emph{on} The Sonc-Cone}
\label{sec:SoncCone}

\subsection{The inequality of arithmetic and geometric means}
Let us repeat a few things from the introduction, using the notation of \S \ref{sec:circuits}.
The barycentric coordinate vector $\bc$ of a simplicial circuit $c$ 
is one of the inputs in the inequality of arithmetic and geometric means \cite[(2.5.5)]{HLP52}, which is equivalent to the following theorem.
\begin{thm}[AGI]
If $c \subset A$ is a simplicial circuit, then the exponential sum
\begin{equation}
\label{eq:agiform}
f(\x) = \<\varphi_A(\x), \bc\>
\end{equation}
is nonnegative on $\R^n$, and it vanishes on the orthogonal complement of the
linear space of the affine span $\Aff(c)$. \qed
\end{thm}
The canonical group action $\R^n \times \R^A \rightarrow \R^A$ acts on the agiform 
\eqref{eq:agiform} by a translation of its singular locus.
That is, the singular locus of the exponential sum 
\begin{equation}
\label{eq:TorusAgiform}
g(\x) = \< \varphi_A(\x-\bw), \bc\>,
\end{equation}
which is equal to the real vanishing locus of $g$, is given by
\begin{equation}
\label{eqn:SingularLocus}
\sing(g) = \bw + \Aff(c)^{\perp}.
\end{equation}

\begin{remark}
\label{rem:Normalization}
The torus action can be used to impose assumptions on the location of the singular loci.
For example, if the exponential sum $f(\x)$ has a unique minimizer $\hat \x$, then it is often assumed
that $\hat \x = 0$. In this work, we impose no such assumptions. On some occasions,
we will use the torus action to impose assumptions on the coefficients of $f(\x)$. 
The action $(\bw, g) \mapsto g(\x - \bw)$ corresponds to component-wise multiplication
of the coefficient vector of $g$ by $\varphi(-\bw)$. If we choose an independent
set $\al_0, \dots, \al_n$ (i.e., the vertices of a simplex), then we can always find $\bw$
such that $g(\x - \bw)$ has coefficients
\begin{equation}
\label{eqn:CoefficientAssumptions}
a_0 = a_1 = \dots = a_n = 1.
\end{equation}
This corresponds to restricting $\R^A$ to the affine subspace defined by \eqref{eqn:CoefficientAssumptions}.
\end{remark}

\begin{example}
The first explicit example of a nonnegative polynomial which is not a sum of squares in the literature is the  \struc{\textit{Motzkin polynomial}} from \cite{Mot67}.  Its inhomogeneous form is
\begin{equation}
\label{eqn:Motzkin}
f(z_1, z_2) = 1+z_1^4z_2^2 + z_1^2z_2^4 - 3\, z_1^2z_2^2.
\end{equation}
This is the (polynomial)
simplicial agiform associated to the circuit
\begin{equation}
\label{eq:MotzkinCircuit}
c = \left[\begin{array}{cccc} 1 & 1 & 1 & 1\\ 0 & 4 & 2 & 2 \\ 0 & 2 & 4 & 2 \end{array}\right]
\quad\text{ and } \quad
\bc = \big[1, 1, 1, -3\big]^\top,
\end{equation}
with singular locus $\bw = (1, 1)$. We refer to $c$ as the \struc{\emph{Motzkin circuit}}.
Notice that the coefficients of $f$ are the coordinates of the vector $\bc \in \ker(A)$.
\label{Example:MotzkinPolynomial}
\end{example}

\begin{remark}
Given an exponential sum $g(\x)$ as in \eqref{eq:TorusAgiform}, a formulaic expression
for the singular locus $\bw$ in terms of the coefficients of $g$ and the barycentric coordinate 
vector $\bc$ was given in \cite[\S3.4]{Iliman:deWolff:Circuits}, when studying circuit polynomials.
They also defined an invariant of a polynomial supported on a circuit, called the
\emph{circuit number}, from which one can immediately read off whether the polynomial
is nonnegative or not.
\end{remark}

\subsection{The sonc-cone and reduced sonc-decompositions}

The most general definition of the sonc-cone 
is the following.

\begin{definition}
\label{def:SoncCone}
A \struc{\emph{sonc-decomposition}} of an exponential sum
$f(\x)$ is a decomposition of $f(\x)$ as a positive combination of 
monomials and nonnegative exponential sums
supported on simplicial circuits.
The \struc{\emph{sonc-cone}} $\struc{\Splus_A} \subset \R^A$ consists of all exponential sums $f \in \R^A$ which admits a sonc-decomposition.
\end{definition}

\begin{remark}
By \eqref{eqn:AgiformSingular}, the vanishing locus of an agiform
is an affine space. In particular, if an exponential sum $f(\x)$ admits a sonc-decomposition,
then its vanishing locus is an intersection of affine spaces and, hence, affine.
To find an example of an exponential sum $f(\x)$ which belongs to the sparse
nonnegativity cone $\Pplus_A$, but which does not belong to the sonc-cone $\Splus_A$, it
suffices to find a nonnegative exponential sum whose vanishing locus is not
an affine space. The simplest example is given by the univariate exponential sum
\[
f(x) = (e^x - 1)^2 \, (e^x - 2)^2,
\]
which belongs to $\R^A$ for $A = \{0, 1, 2, 3, 4\}$.
\end{remark}

\begin{example}
	One exponential sum which is nonnegative but not included in the
	sonc-cone is obtained from the \struc{\emph{Robinson-Polynomial}}
	\begin{align*}
		f(z_1, z_2) \ = \ 1 + z_1^6+z_2^6-(z_1^4z_2^2+z_1^2z_2^4+z_1^4+z_2^4+z_1^2+z_2^2)+3z_1^2z_2^2,
	\end{align*}
	which was the second example of a polynomial that is nonnegative but not a sum of squares  \cite{Robinson}.
	All terms of $f(\bz)$ with negative coefficients have exponent vectors that are located in the boundary of the Newton polytope.
In particular, the point $(2,2)$ is not a vertex of any simplicial circuit contained in $A$.
It follows (see, e.g., \cite{Chandrasekaran:Murray:Wierman}) that $f(\bz)$ admits a sonc-decomposition if and only if the polynomial
\[
g(\bz) = f(\bz) - 3z_1^2z_2^2
\]
admits a sonc-decomposition. But $g(1, 1) = -3$, so that
$g(\bz)$ isn't even nonnegative.
\end{example}

Before we continue, let us state a few important results
which appear---though stated differently---in the contemporary literature.
\begin{thm}[The ``Sonc = Sage'' Theorem]
\label{thm:OneNegativeTerm}
Let $f \in \R^A$ be an exponential sum which has at most one negative coefficient.
Then, $f \in \Pplus_A$ if and only if $f\in \Splus_A$.
\end{thm}

\begin{proof}
This result is in its original form by Reznick \cite[Theorem 7.1]{Rez89},
extended to the general case by Wang \cite[Theorem~1.1]{Wang:AtMostOneNegativeTerm} and later independently by Chandrasekaran, Murray, and Wierman \cite[Theorem 4]{Chandrasekaran:Murray:Wierman}.
This result can also be recovered using the methods developed in this paper.
\end{proof}

\begin{prop}
\label{prop:SoncInclusionTheorem}
If an exponential sum $f\in \R^A$ admits a sonc-decomposition, then $f$ admits a sonc-de\-compo\-sition such that all monomials and all simplicial circuits appearing are contained in $A$.
\end{prop}

\begin{proof}
The statement is the relaxation of \cite[Theorem 2]{Chandrasekaran:Murray:Wierman}.
\end{proof}

\begin{prop}
\label{prop:Codimension2}
Let $A$ be a support set of codimension $m \leq 2$.
Then, $\Pplus_A=\Splus_A$.
\end{prop}

\begin{proof}
Either $A$ is has at most one interior point,
or the Newton polytope $\cN$ is a simplex and $A$ has two non-vertices. Both cases are covered by \cite[Corollary 21]{Chandrasekaran:Murray:Wierman}.
\end{proof}

The exponential sums supported on circuits appearing in Definition~\ref{def:SoncCone} need not be singular. However, there is no loss of generality
in assuming that they are.

\begin{definition}
A \struc{\emph{reduced sonc-decomposition}} of  $f \in \R^A$ is a decomposition
\begin{equation}
\label{eqn:ReznickSum}
f(\x)  = \quad \sum_{c \in C} t_{c} \, \<\varphi_A(\x - \bw_{c}) ,  \bc\> \quad + \quad \sum_{\al \in A} t_{\al} \, e^{\<\x, \al\>}
\end{equation}
where $t_{c}, t_{\al} \geq 0$, and $\bw_{\bc} \in \R^n$.
\end{definition}

Notice that the simplicial agiforms appearing in a reduced sonc-decomposition
are all singular.
Let us generalize a result of Iliman and the second author \cite{Iliman:deWolff:Circuits}.

\begin{lemma}
\label{lem:CircuitSoncDecomposition}
Let $c$ be a simplicial circuit, let $f \in \Pplus_{c}$, 
and let $\al \in c$ be arbitrary. 
Then, there are constants $t, a \geq 0$, and $\bw \in \R^n$, such that
\[
f(\x) = t\,\<\varphi_{c}(\x-\bw), \bc\>  + a\,e^{\<\x, \al\>}. 
\]
\end{lemma}

\begin{proof}
The statement follows from Theorem~\ref{thm:OneNegativeTerm}.
\end{proof}

\begin{prop}
An exponential sum $f \in \R^A$ admits a sonc-decomposition if and only
if it admits a reduced sonc-decomposition.
\end{prop}

\begin{proof}
By Proposition~\ref{prop:SoncInclusionTheorem}, if $f \in \R^A$ admits a sonc-decomposition, then it admits a sonc-decomposition whose terms are contained in $\R^A$. By Lemma~\ref{lem:CircuitSoncDecomposition},
we can assume that all  nonnegative exponential sums supported on circuits appearing in the decomposition are singular. 
If some simplicial circuit $c$ appears twice in the decomposition,
then the sum of the corresponding agiforms is a nonnegative exponential sum supported on $c$, to which we can yet again apply
Lemma~\ref{lem:CircuitSoncDecomposition}.
\end{proof}

Henceforth, all sonc-decomposition in this work are assumed to be reduced.
Notice that \eqref{eqn:ReznickSum} is a parametric representation
of the sonc-cone, on the finite dimensional parameter space
whose coordinates are the scaling factors $\bt$ and the singular loci $\bw_{\bc}$.

\begin{example}
\label{ex:CircuitSoncDecompositions}
Consider the polynomial sonc-decomposition
\[
 f(z) = 2(1 - z)^2 + 2z^2,
\]
where the first terms is a simplicial (polynomial) agiform
supported on the circuit $c = \{0, 1, 2\}$, and the second term is a
monomial with exponent $2$.
The reader can verify that
\[
f(z)  = 2\big(1 - \sqrt{2} z\big)^2 + 4\big(\sqrt{2} - 1\big)z
= (2z - 1)^2 + 1.
\]
That is, $f(z)$ also admits sonc-decompositions where the monomial
term has exponent 1 respectively 0. In particular, there is no minimal
reduced sonc-decomposition of $f(z)$ in the sense that the index set
(which consists of simplicial circuits and exponents) of the composition is minimal.
We show later that a univariate exponential sum has a minimal sonc-decomposition
if and only if it belongs to the boundary of the sonc-cone. 
That said, in many cases we prefer to work with a (reduced) sonc-decomposition
that contains as many terms as possible. For example, we also have that
\[
f(z) = \frac{(14 z - 9)^2}{56}  + \frac{z^2}2 + \frac z2 + \frac{31}{56}.
\]
which is a sonc-decomposition that contains terms
indexed by the circuit $c$ as well as all exponents $\al \in c$.
We say that the \struc{\emph{sonc-support}} of $f(z)$ 
is $\{c, 0, 1, 2\}$. But we are getting ahead of ourselves;
sonc-supports is the topic of the subsequent section.
\end{example}

\section{Sonc-supports}
\label{sec:SoncSupports}

Recall that $C$ denotes the set of all simplicial circuits
contained in the support set $A$.
In general, sonc-decompositions of an exponential sum $f \in \Splus_A$ are not unique.
In particular, the exponential sum $f$ does not determine the the set of monomials and simplicial circuits which appear in
the sonc-decomposition \eqref{eqn:ReznickSum} with a positive coefficient.

\begin{definition}
\label{def:SoncSupport}
The \struc{\emph{sonc-support}} of an exponential sum $f\in \R^A$
is the set
\[
\struc{\ssp(f)} = \big\{\,\br \in A \cup C \,\, \big|\,\, f \text{ admits a sonc-decomposition with } t_{\br} > 0 \,\,\big\}.
\]
We write $\struc{\ssp(f; A)}$ in case we need to specify which support set is considered.
A subset $R\subset A \cup C$ is called a \struc{\textit{sonc-support}} if there exists some $f\in \R^A$
such that $R = R(f)$.
\end{definition}

That is, the sonc-support of $f$ consists of all simplicial circuits and monomials that appear in some sonc-decomposition of $f$.

\begin{example}
Consider the Motzkin polynomial $f$ from Example \ref{Example:MotzkinPolynomial}.
As defined, it is supported on the Motzkin circuit.
As it has a singular point at $(1,1)$, its sonc-support contains no monomials,
so that $R(f)$ consists only of the Motzkin circuit.
\end{example}

Just as the sonc-decomposition \eqref{eqn:ReznickSum} 
has a singular part and a regular part, 
the sonc-support $R(f)$ splits into two parts.
However, we use a slightly more elaborate definition.

\begin{definition}
\label{def:SingularPart}
The \struc{\emph{singular part}} of a sonc-support $R(f)$ is the set
\[
\smsp(f) = \{ \,\bc \in \ssp(f)\cap C\, \mid \, \al \in \bc  \Rightarrow \al \notin \ssp(f)\,\}.
\]
The \struc{\emph{regular part}} of the sonc-support $\ssp(f)$ is the complement $\mosp(f) = \ssp(f)  \setminus \smsp(f)$.
\end{definition}

By Lemma \ref{lem:CircuitSoncDecomposition}, the singular part $S(f)$ consists of all simplicial circuits $c \in R(f)$ such that any exponential sum supported on 
$c$ appearing in a sonc-decomposition of $f$ is singular.
For example, in case of the Motzkin polynomial, the sonc-support consists only of a singular part, namely the Motzkin circuit; the regular part is empty.
Notice that the Motzkin polynomial is contained in the boundary of the sonc-cone (and the nonnegativity cone). The intuition behind the following proposition is that
 the boundary of the sonc-cone is characterized by that $S(f) \neq \varnothing$.

\begin{prop}
\label{pro:InteriorPositiveCombination}
An exponential sum $f \in \R^A$ lies in the interior of the sonc-cone 
$\Splus_A$ if and only if $R(f) = A \cup C$.
\end{prop}

\begin{proof}
For each $\br \in A \cup C$, fix an agiform (or a monomial) $g_{\br} \in \R^{\br}$.
Since the sonc-cone is full-dimensional,
if $f$ belongs to the interior of $\Splus_A$, then so does the exponential sum
\[
h(\x) = f(\x) - \varepsilon \sum_{\br\in A \cup C} g_{\br}(\x),
\]
for $\varepsilon > 0$ sufficiently small.
By definition, $h$ admits a sonc-decomposition. To write $f$ as a positive combination of the agiforms
$g_{\br}$, it suffices to move the sum appearing in the above equality to the left hand side. The converse statement follows from that $A \subset A \cup C$.
\end{proof}

\begin{example}
\label{Example:MotzkinStrictPositivesoncComplexes}
Let $f$ be the Motzkin polynomial from \eqref{eqn:Motzkin}, and consider the positive polynomial
\begin{equation}
\label{eqn:Motskin2}
g(z_1, z_2) = f(z_1, z_2) + \frac1{10}\, {z_1^2z_2^2},
\end{equation} 
which belongs to the interior of the sonc-cone.
The expression \eqref{eqn:Motskin2} is a reduced sonc decomposition of $g$,
including terms indexed by the Motzkin circuit $c$ and the point $(2,2)$. 
The reader can verify that
\begin{equation}
\label{eqn:Motskin3}
	 g(z_1, z_2) \ = \ \Big(1 + z_1^4z_2^2 + \frac{29^3}{30^3} \cdot z_1^2z_2^4 - \frac{29}{10} \cdot z_1^2z_2^2\Big) + \frac1{10} \cdot z_1^4z_2^2.
\end{equation}
One can check  (e.g., using \cite[Theorem 1.1]{Iliman:deWolff:Circuits}) that the first term in \eqref{eqn:Motskin3} is a singular nonnegative agiform.
Hence, \eqref{eqn:Motskin3} implies that $(4,2) \in R(g)$.
Using analogue arguments, one can show that $(2,4), (0,0) \in R(g)$.
Hence, as stated in Proposition \ref{pro:InteriorPositiveCombination},
$R(g) = A \cup C$, where $A = c$ is the Motzkin circuit.
\end{example}

Fix a sonc-support $R$. Then, each reduced sonc-decomposition of an exponential sum $f(\x)$ whose sons-support is $R$, is determined by the scaling factors $\bt$, and 
by the parameters $\bw$, se \eqref{eqn:ReznickSum}.
By the following lemma, whether $f(\x)$ belongs to
the boundary of the sonc-cone or not, depends only on the 
relative positions of the singular loci $\bw$.

\begin{lemma}
\label{lem:LocationOfSingularStrata}
Let $R \subset A\cup C$, and assume that $f$ admits a sonc-decomposition
$f(\x) = \sum_{\br \in R} t_{\br} \, g_{\br}(\x)$.
Then, $f$ belongs to the boundary of $\Splus_A$ if and only if for all $s_{\br} > 0$, the exponential sum 
$h(\x) = \sum_{{\br} \in R} s_{\br} \, g_{\br}(\x)$
belongs to the boundary of the sonc-cone.
\end{lemma}

\begin{proof}
Let $\varepsilon$ be defined by that $\min_{{\br} \in R}(t_{\br})= \varepsilon\max_{{\br} \in R}(s_{\br})$. We have that
\[
f(\x) = \varepsilon h(\x)  + \sum_{{\br} \in R} (t_{\br} - \varepsilon s_{\br}) g_{\br}(\x).
\]
As the last sum has nonnegative coefficients, we conclude that $R(h) \subset R(f)$.
The reverse inclusion follows by symmetry. The lemma now follows
Proposition~\ref{pro:InteriorPositiveCombination}, 
which says that an exponential sum belongs to the interior of the sonc-cone if and only if
$R(f) =A\cup C$.
\end{proof}

\subsection{Exponential sums with a unique minimizer}
The following propositions, concerning the case that $R(f) = C$,
are technical yet crucial for our further investigation.

\begin{prop}
\label{pro:AllCircuitsSingularPoint}
Assume that $A$ has a relative interior point.
Let $f \in \R^A$ be such that $\ssp(f) = C$. 
Then, $f$ has a singular point.
\end{prop}

\begin{proof}
Let $\al_0$ be a relative interior point of $A$. We use induction over the number of points of $A$ which are not vertices of the Newton polytope $\cN$.
The base case if when all points of $A\setminus \{\al_0\}$ are vertices
of $\cN$. Then, $f$ has at most one negative term.
Since $\ssp(f) \neq A \cup C$, we have that $f$ belongs to the boundary of $\Splus_A$, so $f$ belongs to the boundary of $\Pplus_A$ by Theorem~\ref{thm:OneNegativeTerm}.
Hence, $f$ has a singular point.

Now assume that there is a point $\al_1 \in A \setminus \{\al_0\}$ which is
not a vertex of $\cN$.
By Lemma~\ref{lem:SimplexCircuitConstruction},
there exists $c_1 \in C$ with $\al_1$
as an interior point, and there exists $c_0\in C$ with $\al_1$ as a vertex and 
$\al_0$ as an interior point. 
Let us enumerate the remaining simplicial circuit contained in $A$
by $c_2, \dots, c_k$. Consider a sonc-decomposition
\begin{equation}
\label{eqn:Prop4.7.1}
f(\x) = t_0\, g_0(\x) + \dots + t_kg_k(\x),
\end{equation}
where $g_i(\x)$ is an agiform supported on $c_i$, and $t_i > 0$.
Since the coefficient of $e^{\<\x, \al_1\>}$ is negative in $g_1(\x)$ and positive in $g_0(\x)$,
we can find $s_0, s_1 > 0$ such that  the coefficient of $e^{\<\x, \al_1\>}$ in the exponential sum
\begin{equation}
\label{eqn:Prop4.7.2}
h(\x) = f(\x) + s_0\, g_0(\x) + s_1 \,g_{1}(\x)
\end{equation}
vanishes. Let $A' = A\setminus\{\al_1\}$, and let $C'$ denote the set of simplicial 
circuits contained in $A'$. It follows from \eqref{eqn:Prop4.7.1} and \eqref{eqn:Prop4.7.2} that $\ssp(h; A') = C'$, since any circuit of $C'$ appears in the sonc-decomposition \eqref{eqn:Prop4.7.1}.
Since  $A'$ has a strictly smaller cardinality than $A$, we can apply the induction 
hypothesis and conclude that $h(\x)$ has a singular point $\bw$.
Then,
\[
0  = h(\bw) \geq f(\bw) \geq 0,
\]
which implies that $\bw$ is a singular point of $f(\x)$.
\end{proof}

\begin{prop}
\label{pro:DiscSoncSuppWhenAddingPoints}
Assume that $A$ has a relative interior point, and let $f\in \R^A$ be such that $\ssp(f; A) = C$. Let $\al_{d+1} \in \cN \setminus A$,
and set $\hat A = A \cup \{\al_{d+1}\}$.
Then, $\ssp(f; \hat A) = \hat C$,
where $\hat C$ denotes the set of simplicial circuits contained in $\hat A$.
\end{prop}

\begin{proof}
Let $C = \{c_1, \dots c_k\}$.
By Proposition~\ref{pro:AllCircuitsSingularPoint},
the exponential sum $f(\x)$ has a singular point, which (using the group action)
can be assumed to be $\x = (0,\dots, 0)$.
Then, there exists a sonc-decomposition
\[
f(\x) = \sum_{i=1}^k \, \< \varphi_A(\x), t_i \bc_i \>.
\]

Let $\al_0$ denote the relative interior point of $A$. 
On the one hand, by Lemma~\ref{lem:SimplexCircuitConstruction}, we find a simplicial circuit $\bd_0 \subset \hat A$ containing $\al_{d+1}$ as a vertex and $\al_0$ as an interior point. 
On the other hand, $\al_{d+1}$ is not a vertex of $\cN$,
so by Lemma~\ref{lem:SimplexCircuitConstruction}, there is a simplicial circuit $\bd_1 \subset A_0 = A \setminus\{\al_0\}$ containing $\al_{d+1}$ as its interior point.
Then, there are positive constants $s_0$ and $s_1$ such that
the entry corresponding to $\al_{d+1}$ in the linear combination
$s_0\,\bd_0 + s_1 \,\bd_1$ vanishes.
It follows that there are real numbers $r_1, \dots, r_k$ such that
\[
s_0\,\bd_0 + s_1 \,\bd_1 =  r_1 \bc_1 + \dots + r_k \bc_k
\]
and, by choosing $\varepsilon > 0$ sufficiently small (so that $t_i-\varepsilon r_i > 0$),
\[
f(\x) = \,\< \varphi_A(\x),\varepsilon s_0\bd_0\> +  \,\< \varphi_A(\x), \varepsilon s_1\bd_1\> + \sum_{i=1}^k \, \< \varphi_A(\x), (t_i-\varepsilon r_i) \bc_i \>,
\]
is a sonc-decomposition of $f(\x)$.
We conclude that $\bd_0, \bd_1 \in \ssp(f; \hat A)$.
Since $A$ has a relative interior point, $\cR_A$ is full-dimensional by Proposition~\ref{prop:BasisSimplexCircuits}, 
so that $\ssp(f; A)$ contains a basis of $\ker(A)$.
Hence, $\ssp(f; \hat A)$, which contains both $\ssp(f; A)$ and $\bd_0$,
contains a basis of $\ker(\hat A)$.
That is, for any simplicial circuit $\bd \subset \hat A$ we have that $\bd$
can be written as a linear combination of the vectors $\bd_0$ and $\bc_1, \dots, \bc_k$.
The same trick we used to show that $\bd_0$ and $\bd_1$ belong to $\ssp(f; \hat A)$
gives that $\bd \in \ssp(f; \hat A)$.
\end{proof}

\begin{example}
Consider the univariate (polynomial) agiform
\[
g_1(z) = 1 - 3 z^2 + 2 z^3.
\]
As a polynomial supported on $c_1 = \{0, 2, 3\}$, we have that
$\smsp(g_1; c_1) = \{c_1\}$. 
Now consider $g_1$ as a polynomial supported on $A = \{0, 1, 2, 3\}$.
Let $g_i$ is supported on the circuit $c_i = A \setminus\{i\}$ with a singular
point at $(1,1)$. That is,
\begin{align*}
	g_0(z) \ = \ z - 2z^2 + z^3, \quad g_2(z) \ = \ 2 - 3z + z^3, \quad
	\text{and} \quad  g_3(z) \ = \ 1 - 2z + z^2.
\end{align*}
It is straightforward to verify that
 \[
4 \,g(z) = 5\, g_0(z) + g_1(z) + g_2(z) + g_3(z).
\]
Hence, 
$\smsp(g; A) = \{c_0, c_1, c_2, c_3\}$.
In this manner, Proposition~\ref{pro:DiscSoncSuppWhenAddingPoints}
describes conditions under which we can control how sonc-supports change
when we add points to the support set $A$.
We make use of this trick in several of the subsequent proofs.
\end{example}
 
\subsection{Truncations and the location of singular loci}
Let $A'$ be a subset of $A$, with set of simplicial circuits $C'$.
We define the \struc{\emph{truncation}} of a sonc-decomposition 
\[
 \sum_{\br \in A \cup C} t_{\br}\, g_{\br}(\x)
\]
to the subset $A'$ to be the sonc-decomposition
\[
\sum_{\br \in A'\cup C'} t_{\br}\, g_{\br}(\x).
\]
It is important to note that the truncation depends on the sonc-decomposition. That is,
truncation is \emph{not} a well-defined operator on $\R^A$.
Rather, it should be seen as en operation on the parameter space
of the space of all sonc-decompositions with a given sonc-support.
The most interesting truncations are to \struc{\emph{extremal}} subsets of $A$; 
that is, subsets of the form $A_F = A \cap F$, where $F$ is a face of the Newton polytope $\cN(A)$. 

\begin{lemma}
\label{lem:TruncationsSingularLocus}
Assume that $A$ has a relative interior point and that $f \in \R^A$ has
$\ssp(f) = C$. 
Let $F \preccurlyeq \cN(A)$ be a face of $\cN(A)$. 
Let $g$ be a truncation of a sonc-decomposition
of $f$ to the face $F$, such that the Newton polytope
of $g$ is $\cN(A\cap F)$. Then,
\[
\sing(g) = \sing(f) + \Aff(F)^\perp.
\]
\end{lemma}

\begin{proof}
We can assume that $\cN(A)$ is full dimensional, in which case $\sing(f) = \bw$ is a point.
By Proposition \ref{pro:AllCircuitsSingularPoint}, 
we have that $\bw = \sing(f) \subset \sing(g)$.
Hence, using \eqref{eqn:SingularLocus}, we find that
$\sing(g) = \bw + \Aff(F)^{\perp}$.
\end{proof}

\begin{example}
	Consider the polynomial
	\begin{equation}
		f(z_1,z_2) = g(z_1, z_2) + h(z_1,z_2), \label{Equations:MotzkinVariationExampleTruncation}
	\end{equation}
	where $g(z_1, z_2) = z_1^2 - 2z_1 + 1$ and $h(z_1,z_2)$ is the Motzkin polynomial \eqref{eqn:Motzkin}.
	Both, $g(z_1, z_2)$ and $h(z_1,z_2)$ are nonnegative singular simplicial agiforms with a singular point
	at $(1, 1)$.
	Notice that $g$ is the truncation of $f$ to the face $F \preccurlyeq \cN$ which is contained in the
	first coordinate axis and that, since $g$ is constant with respect to $z_2$,
	\[
	  \sing(g) = \{(1,s) : s \in \R\} = (1,1) + \Aff(F)^\perp.\qedhere
	\]
\end{example}

\begin{thm}
\label{thm:OverlappingConstraint}
For $i = 1, 2$, let $A_i$ be a support set which contains a relative interior point, and
let $C_i$ denote the set of simplicial circuits contained in $A_i$. 
Let $A = A_1 \cup A_2$, and let $C$ denote the set of simplicial circuits contained in $A$. 
Let $P = \cN(A_1) \cap \cN(A_2)$ and, for $i = 1,2$, let $F_i$
denote the minimal face of $\cN(A_i)$ containing $P$.
Let $f_i$ be supported on $A_i$ with $\ssp(f_i; A_i) = C_i$.
Then, the exponential sum
\[
f(\x) = f_1(\x) + f_2(\x)
\] 
either has $\ssp(f; A) = A \cup C$, or there is a  translate of the normal space of $\Aff(F_1 \cup F_2)$
which intersects both $\sing(f_1)$ and $\sing(f_2)$.
\end{thm}

\begin{proof}
If $P = \varnothing$ or if $P$ is a vertex of both
$\cN(A_1)$ and $\cN(A_2)$, then there is nothing to
prove. There are two more cases: the case that one of $F_1$ and $F_2$ is zero-dimensional,
and the case that neither $F_1$ nor $F_2$ is zero-dimensional.

First, note that if $\al \in P$ is a relative interior point of $F_i$, then by Proposition~\ref{pro:DiscSoncSuppWhenAddingPoints}, there is no loss of generality in assuming that there is a circuit $c \subset A_i\cap F_i$, with $\al$ as its relative interior point, such that $\dim(c)= \dim(F_i)$. 
Then, by Lemma~\ref{lem:TruncationsSingularLocus}, there is no loss of generality
in replacing $f_i$ by a simplicial agiform $g_i$, supported on $\bc$, 
which appears in some sonc-decomposition of $f_i$.

Second, we note that if $\al \in A_1 \cap A_2$, then 
$\al \in \ssp(f; A)$ implies that $\ssp(f; A) = A \cup C$. Indeed,
Lemmas~\ref{lem:CircuitSoncDecomposition} and \ref{lem:SimplexCircuitConstruction}
give that: if $\al \in \ssp(f_i; A_i)$, then the relative interior point of $A_i$ is contained
in $\ssp(f_i; A_i)$, which in turn implies that $A_i \subset \ssp(f_i; A_i)$
for $i=1, 2$. 

Consider now the case that neither $F_1$ nor $F_2$ is zero-dimensional.
Let $\al$ be a relative interior point of $P$, so that $\al$ is a relative interior point of both $F_1$ and $F_2$.
As above, we can reduce to the case of two simplicial agiforms $g_1$ and $g_2$
with supports $\bc_1 \subset F_1$ respectively $\bc_2 \subset F_2$,
both with $\al$ as a relative interior point.
It suffices to show that either $g_1$ and $g_2$ have a common singular
point, or the exponential sum
$g(\x) = g_1(\x) + g_2(\x)$
has $\al \in \ssp(g; \bc_1 \cup \bc_2)$.
But the exponential sum $g(\x)$ has exactly one negative term. Hence, if $g$ is strictly positive
(i.e., if $g_1$ and $g_2$ does not have a common singular point),
then $\al \in \ssp(g)$ by Theorem~\ref{thm:OneNegativeTerm}.

Consider now the case that $F_2 = P = \{\al\}$ is a vertex of $\cN(A_2)$, 
while it is a relative interior point of $F_1$.
As above, we can assume that $F_1 = \cN(A_1)$.
Then, $\Aff(F_1 \cup F_2) = \Aff(F_1)$, so by \eqref{eqn:SingularLocus}
it remains to show that either $\al \in \ssp(f)$ or $f_1$ and $f_2$ have a common singular
point. By Proposition~\ref{pro:DiscSoncSuppWhenAddingPoints}, there is no loss of generality in 
assuming that $A_2$ contains a circuit $\bc_2$ with $\al$ as a vertex such that 
$\dim \bc_2 = \dim \cN(A_2)$.
Hence, we can reduce to two simplicial agiforms $g_1$ and $g_2$.
Since the monomial $e^{\<\x, \al\>}$ has coefficients of different signs in
$g_1$ and $g_2$, then there exists $t_1, t_2 > 0$
such that the coefficient of $e^{\<\x, \al\>}$
in $t_1g_1(\x) + t_2g_2(\x)$ vanishes.
Then, $t_1g_1(\x) + t_2g_2(\x)$ has one negative term,
 and we can apply Theorem~\ref{thm:OneNegativeTerm}
as in the previous case. 
\end{proof}

\begin{example}
Let $h(z_1,z_2)$ be the Motzkin polynomial and let
\[
g_1(z_3) = 1 - 2z_3 + z_3^2
\quad\text{and}\quad
g_2(z_4) = 1 - z_4 + \frac14 z_4^2,
\]
so that $\sing(g_1) = \{z_3 \, | \, z_3 = 1\}$ and $\sing(g_2) = \{z_4 \,|\,  z_4 = 2\}$. Let $\bz = (z_1, z_2, z_3, z_4)$ and introduce
\begin{align*}
f_1(\bz) & = h(z_1,z_2) + g_1(z_3) \quad \text{and}\\
f_2(\bz) & = h(z_1,z_2) + g_2(z_4),
\end{align*}
so that $\sing(f_1) = \{\bz \, |\,  z_1 = z_2 = z_3 = 1\}$ and $\sing(f_2) = \{\bz \, |\,  z_1 = z_2 = 1 \text{ and } z_4 = 2\}$.
Let $A_1$ and $A_2$ be the support sets of $f_1$ and $f_2$ respectively.
Notice that the only two simplicial circuits contained in $A_1$ are the Motzkin circuit
and the support set of $g_1$. Similarly, the only simplicial circuits contained in $A_2$
are the Motzkin circuit and the support set of $g_2$.
Since $f_1$ and $f_2$ are both singular, neither admits a sonc-decomposition that
contains monomials. It follows that $R(f_1) = C_1$ and $R(f_2) = C_2$, where $C_1$
and $C_2$ denote the sets of simplicial circuits contained in $A_1$ respectively $A_2$.
Now, let
\[
f(\bz) = f_1(\bz) + f_2(\bz),
\]
with support set $A$ and associated set of simplicial circuits $C$. 
Notice that $f(1, 1, 1, 2) = 0$, so that $R(f; A) \neq A \cup C$.
We have that $\cN(A_1) \cap \cN(A_2) = \cN(c)$.
Hence, by Theorem \ref{thm:OverlappingConstraint},
there exists a translate $L$ of the normal space of $\cN(c)$ intersecting $\sing(f_1)$ and $\sing(f_2)$.
Indeed, $L = \{\bz \, | \, z_1 = z_2 = 1\}$.
\end{example}

\subsection{Obstructions}
\label{ssec:Obstructions}
We now turn to the regular sonc-support $M(f)$. By Definition~\ref{def:SingularPart},
no $\al \in M(f)$ can be contained in a circuit $c$ which belongs to the
singular sonc-support $S(f)$. In fact, $S(f)$ imposes further \emph{obstructions} on $M(f)$:
If $c \in S(f)$, then no $\al \in M(f)$ can be contained in the affine space of $c$.

\begin{corollary}
\label{cor:RegularSoncSupporIntersection}
Let $f \in \R^A$. Then, for each $c \in S(f)$, the intersection of $\cN(A \cap M(f))$ 
with the affine span $\Aff(c)$ is empty. 
In particular, if $c \in S(f)$ is such that $\cN(c)$ is full-dimensional,
then $M(f) = \varnothing$.
\end{corollary}

\begin{proof}
Assume, on the contrary, that there exists a point $\al \in \Aff(c) \cap \cN(A \cap M(f))$.
Since $\al \in \cN(A \cap M(f))$, we can find a simplicial circuit $c'$ with vertices in 
$A \cap M(f)$ and $\al$ as its relative interior point. 
By definition, $f(\bz)$ has a sonc-decomposition that contains all terms 
with exponents in $A \cap M(f)$ with positive coefficients.
It follows that $f(\bz)$ admits a sonc decomposition that
contains an agiform supported on $c'$ and an exponential monomial
$a\,e^{\<\x, \al\>}$ with a positive coefficient. By assumption, the sonc-decomposition can
be chosen such that it contains an agiform $g_c(\x)$ supported on $c$.

The exponential sum $h(\x) = g_c(\x) + a\,e^{\<\x, \al\>}$, 
which appears in a sonc-decomposition of $f(\bz)$,
contains an exponential monomial with a positive coefficient and, hence, it is strictly positive.
That is, $h(\x)$ belongs to the interior of the nonnegativity cone $\Pplus_{A'}$,
where $A' = c \cup \{\al_1\}$. But $A'$ has codimension two and, hence, it follows from Proposition~\ref{prop:Codimension2} that $\Splus_{A'} =  \Pplus_{A'}$.
Hence, $h(\x)$ is contained in the interior of $\Splus_{A'}$. By Proposition \ref{pro:InteriorPositiveCombination}, $h(\x)$ admits a sonc-decomposition
which contains exponential monomials for all elements of $c$.
That is, $c \subset M(f)$, contradicting that $c \in S(f)$.
\end{proof}

\begin{example}
Consider the agiform $f(z_1, z_2) = 1 - 2z_1 + z_1^2$, which is supported on $c = \{(0,0),(1,0),(2,0)\}$.
Let
\[
g_1(z_1,z_2) = f(z_1) + z_1z_2^2.
\]
Since $g_1(1,0) = 0$, there is no sonc-decomposition of $g_1$ that contains a monomial 
which is constant with respect to $z_2$. Hence, $c \in S(g)$ and $(1,2) \in M(g)$.
Notice that $(1,2)$ is not contained in the affine span of the circuit $c$.
Now consider
\[
g_2(z_1,z_2) = f(z_1) + z_1^3.
\]
The exponent $(3,0)$ lies in the affine span of $c$.
By Corollary \ref{cor:RegularSoncSupporIntersection} it follows that $S(g_2) = \emptyset$.
Indeed, we have that
\[
g_2(z_1,z_2) = \Big(\frac12 -z + \frac12z^2\Big) + \Big(\frac12 - z + \frac{16}{27} z^3\Big) + \frac12 z^2 + \frac{11}{27} z^3.
\]
Since any simplicial circuit in the support of $g_2$ contains either $(2,0)$ or $(3,0)$,
we conclude that $S(g_2) = \emptyset$.
\end{example}

\section{The univariate case}
\label{sec:Univariate}

We represent a univariate support set as
\begin{equation}
\label{eqn:UnivariateA}
A = \big\{\alpha_0, \alpha_1, \dots, \alpha_d\big\}
\end{equation}
where $\alpha_0 < \alpha_1 < \dots < \alpha_d$,
and we denote the coefficients by $a_0, \dots, a_d$.
By Remark \ref{rem:Normalization}, there is no loss of generality in restricting
to the affine space $a_0 = a_d = 1$ in $\R^A$.
It follows from Corollary~\ref{cor:RegularSoncSupporIntersection},
and the fact that each circuit $c \subset A$ is full-dimensional,
that any exponential sum $f(x)$ which belongs to the boundary of the sonc-cone
has empty regular sonc-support.
We showed in \S \ref{sec:ReznickCone}
that $\cR_A$ is a full-dimensional and simple polyhedral cone, generated by the 
vectors $\bc_i$ associated to the circuits
\begin{equation}
	c_i = \{\alpha_{i-1}, \alpha_i, \alpha_{i+1}\},
\quad \text{for } i = 1, \dots, m. \label{Equation:DistinguishedUnivariateCircuits}
\end{equation}
In particular, the faces of $\cR_A$ correspond bijectively to index sets
\begin{equation}
\label{eqn:UnivariateSimplicialSoncSupport}
I = \{i_1, \dots, i_k\} \subset \{1, \dots, m\}.
\end{equation}
Note that only index sets containing $1$ and $m$ give a piece of the
boundary of the sonc-cone which intersects the
affine space $a_0 = a_d = 1$.

\begin{example}
\label{ex:DimensionOneBoundaryLemma}
The following modest example is, in fact, one of the main building blocks
for our proof of Theorem \ref{Thm:UnivariateUniqueTropical} and, in effect, Theorem \ref{thm:Main}.
Consider the integral support set $A = \{0,1,2,3,4\}$, and the two simplicial circuits 
$c_1 = \{0, 1, 2\}$ and $c_3 = \{2, 3, 4\}$.
For $i = 1, 2, 3$, let $g_i(z)$ be a (polynomial) simplicial agiform supported on $c_i$,
with singular loci $w_i$, and consider a polynomial
\begin{equation}
\label{eqn:UniExample1}
f(z) = t_1g_1(z) +t_3 g_3(z),
\end{equation}
where $t_1, t_2 > 0$. We claim that $f$ belongs to the boundary of $\Splus_A$ if an only if $w_1 \leq w_3$.
By Proposition \ref{pro:AllCircuitsSingularPoint}, we can assume that
$t_1 = t_3 = 1$.
Let us write
\[
g_1(z) = 1 - 2v_1 z + v_1^2 z^2 = (1-v_1z)^2 
\quad \text{ and } \quad
g_3(z) = z^2\left(v_3^2 - 2v_3 z + z^2\right) = z^2(v_3 - z)^2,
\]
where $v_1 = w_1^{-1}$ and $v_3 = w_3$. In particular, $w_1 \leq w_3$ is equivalent to that $v_1v_3 \geq 1$.
Proposition~\ref{pro:InteriorPositiveCombination} implies that the
polynomial $f(z)$ belongs to the interior
of $\Splus_A$ if and only if there is a decomposition
\begin{equation}
\label{eqn:UniExample2}
f(z) = h_1(z) + h_2(z) + h_3(z)
\end{equation}
where $h_1$ and $h_3$ are agiforms supported on $c_1$ respectively $c_3$, and 
$h_2(z) = b_0z - b_1z^2 + b_2 z^3$
is positive on $\R_+$.
Substituting \eqref{eqn:UniExample1} for the left hand side of \eqref{eqn:UniExample2}
and identifying coefficients, we find that there exists a $\boldsymbol{\zeta} = (\zeta_1,\zeta_2)$.
 such that
\[
h_1(z) = g_1(\zeta_1 z) 
\quad \text{ and } \quad
h_3(z) = \zeta_2^4\, g_3(z/ \zeta_2)
\]
Solving \eqref{eqn:UniExample2} for $b_0, b_1, $ and $b_2$, we obtain that
\[
b_0(\boldsymbol{\zeta}) = 2v_1(\zeta_1-1), \quad
b_1(\boldsymbol{\zeta})= v_1^2(\zeta_1^2-1) + v_3^2(\zeta_2^2 -1), \quad \text{and} \quad
b_2(\boldsymbol{\zeta}) = 2v_3(\zeta_2-1).
\]
In particular, if $h_2(z)$ is positive on $\R_+$ then $\boldsymbol{\zeta} > 1$ component-wise,
which in turn implies that $b_1 > 0$. Hence, $f(z)$ belongs to the interior of
$\Splus_A$ if and only we can find $\boldsymbol{\zeta} > 1$ such that
the discriminant of $h_2(z)$ is positive. Let us denote said discriminant by $p(\boldsymbol{\zeta})$,
so that
\[
p(\boldsymbol{\zeta}) = 4 b_0(\boldsymbol{\zeta}) b_2(\boldsymbol{\zeta}) - b_1(\boldsymbol{\zeta})^2.
\]
Let $q(\boldsymbol{\xi}) = p(1+\boldsymbol{\xi})$.
After expanding, notice that $q(\boldsymbol{\xi})$ contains exactly one monomial with a positive coefficient. Its exponent vector
belongs to a facet $F$ of the Newton polytope of $q$. In particular, there exists $\boldsymbol{\xi} > 0$ such that $q(\boldsymbol{\xi})$
is positive if and only if there exists a $\boldsymbol{\xi} > 0$ such that the extremal polynomial
\[
 q_F(\boldsymbol{\xi}) =  8v_1v_3 (2-v_1 v_3) \xi_1 \xi_2 -4 v_1^4 \xi_1^2  - 4 v_3^4 \xi_2^2
\]
 is positive. This is the case if and only if  $v_1v_3 < 1$. 
\end{example}

\begin{thm}
\label{Thm:UnivariateUniqueTropical}
Let $f \in \R^A$ be such that the singular sonc-support $S(f)$ is given by the index
set $I = \{i_1, \dots, i_k\}$. Then, $f$ has a unique sonc-decomposition
\begin{equation}
\label{eqn:UnivariateBoundarySoncDecomposition}
f(x) = g_{i_1}(x) + \dots + g_{i_k}(x),
\end{equation}
where $g_i(x)$ is supported on $c_{i}$,
with a singular points at $w_i$,
and where $w_{i_1} \leq \dots \leq w_{i_k}$.
\end{thm}

\begin{proof}
To see that the sonc-decomposition is unique, notice that
the singular loci of the agiforms $g_i(x)$ are uniquely defined by 
Lemma~\ref{lem:CircuitSoncDecomposition}. Hence, if $f(x)$ admits two distinct sonc-decompositions,
then there is a linear combination of  $g_{i_1}, \dots, g_{i_k}$ which vanishes identically, yielding a contradiction, as  these agiforms are linearly independent over $\R$.

To prove that $w_{i_1} \leq \dots \leq w_{i_k}$, it suffices to consider a sum of two minimal simplicial agiforms
$g_{i_1}$ and $g_{i_2}$. There are three cases to consider, depending on 
the dimension of the intersection
\[
P = \cN(c_{i_1}) \cap \cN(c_{i_{2}}).
\]

\smallskip
\underline{Case 1}: If $P$ is a line segment, then
$w_{i_1} = w_{i_{2}}$ by Theorem~\ref{thm:OverlappingConstraint}.

\smallskip
\underline{Case 2}: Consider the case that $P$ consists of a single point $\alpha_j$.
There is no loss of generality in assuming that
\[
A = \{\alpha_0, \alpha_1, \alpha_2, \alpha_3, \alpha_4\},
\]
and that $j = 2$ (so that $i_1 = 1$ and $i_2 = 3$). Choose $\varepsilon > 0$ such that
\[
[\alpha_2 - 2 \varepsilon, \alpha_2 + 2 \varepsilon] \subset [\alpha_{1}, \alpha_{3}].
\]
Let us consider $f(x)$ as an exponential sum supported in the set
\[
\hat A = A \cup \{\alpha_2  -2\varepsilon, \, \alpha_2 - \varepsilon, \, \alpha_2  +  \varepsilon, \, \alpha_2 + 2 \varepsilon\}.
\]
By Proposition~\ref{pro:DiscSoncSuppWhenAddingPoints}, we can write $g_{1}(x)$ as a
positive combination of agiforms supported on the three minimal circuits
contained in $\{\alpha_{0}, \alpha_{1}, \alpha_2- 2 \varepsilon, \alpha_2 - \varepsilon, \alpha_2\}$,
and we can write $g_{3}(x)$ as a
positive combination of agiforms supported on the three minimal circuits
contained in $\{\alpha_2, \alpha_2+ \varepsilon, \alpha_2 + 2 \varepsilon, \alpha_{3}, \alpha_{4}\}$.
That is, there is no loss of generality in assuming that
\[
\alpha_{2 + k} = \alpha_2 + k\varepsilon, \quad k = -2, -1, 0, 1, 2.
\]
Hence, applying an affine transformation in $x$,
we can assume that $c_{i_1}\cup c_{i_2}$ is 
the support set from Example~\ref{ex:DimensionOneBoundaryLemma}.
It follows from said example that $w_{i_1} \leq w_{i_{2}}$.

\smallskip
\underline{Case 3}:
The last case is when there is an open line segment
separating $\cN(c_{i_1})$ and  $\cN(c_{i_{2}})$.
There is no loss of generality in assuming that
\[
A = \big\{ \alpha_0, \alpha_1, \alpha_2, \alpha_3, \alpha_4, \alpha_5, \alpha_6 \big\}
\]
and that $i_1 = 1$ and $i_2 = 5$. That is, there is no loss of generality
in assuming that $A$ contains a point $\alpha_3$ located
in the line segment separating $\cN(c_{i_1})$ and  $\cN(c_{i_{2}})$.

Let $g_1$ and $g_5$ be minimal simplicial agiforms with singular loci $w_1$ and $w_5$ respectively.
Assume that the exponential sum 
\[
f(x) = g_1(x) + g_5(x)
\]
belongs to the boundary of the sonc-cone $\Splus_A$. We need to show that
$w_1 \leq w_5$.

Since $f(x)$ belongs to the boundary of the sonc-cone, and since the leading and 
the constant term of $f$ have nonvanishing coefficients, Corollary~\ref{cor:RegularSoncSupporIntersection} implies
that the regular sonc-support of $f$ is empty. It follows that the exponential sum
\[
p_{\varepsilon}(x) = f(x) - \varepsilon \,e^{x \alpha_3}
\] 
does not belong to the sonc-cone. On the other hand, if $g_3$ denotes some simplicial agiform
supported on the minimal circuit $c_3 = \{\alpha_2, \alpha_3, \alpha_4\}$, then
\[
q_{\varepsilon}(x) = f(x) + \varepsilon \,g_3(x)
\] 
belongs to the sonc-cone. It follows that there is a convex combination
\begin{equation}
\label{eq:ApproximatingSequence}
h_\varepsilon(x) = s_1(\varepsilon) \, p_{\varepsilon}(x) + s_2(\varepsilon)  q_{\varepsilon}(x)  = f(x)  - \varepsilon \,s_1(\varepsilon) \, e^{x \alpha_3} + \varepsilon\, s_2(\varepsilon)  \,g_3(x),
\end{equation}
where $s_1(\varepsilon) + s_2(\varepsilon)  = 1$, which belongs to the boundary of the sonc-cone.
Note that this is not a sonc-decomposition of $h_\varepsilon(x)$ (since one coefficient is negative),
and that the regular sonc-support of $h_\varepsilon(x)$
is empty.

By the first paragraph of this proof, there exists a unique sonc-decomposition
\[
h_\varepsilon(x) = \sum_{i=1}^5 t_{i}(\varepsilon)\, \<\varphi_A(x-w_i(\varepsilon)), \bc_i\>,
\]
using the minimal circuit $c_1, \dots, c_5$. Since the constant and leading coefficient appears in $h_\varepsilon(x)$ with a positive coefficient, we have that $t_1(\varepsilon), t_5(\varepsilon) > 0$.
Since the monomial $e^{x \alpha_3}$ appears in $h_\varepsilon(x)$ with a negative coefficient,
we have that $t_3(\varepsilon) > 0$. Hence, the first two cases imply that
\[
w_1(\varepsilon) \leq w_3(\varepsilon) \leq w_5(\varepsilon)
\]
for all $\varepsilon > 0$. Taking limits as $\varepsilon \rightarrow 0$,
we obtain that $w_1 \leq w_5$.
\end{proof}

\section{Tropical arrangements of singular loci}
\label{sec:TropicalArrangements}

In this section we show that the singular loci of the simplicial agiforms appearing
in a sonc-decomposition, in the multivariate case, are arranged according to a tropical hypersurface. Let us first introduce the basic elements of tropical geometry.

A \struc{\textit{tropical exponential sum}} is a max-plus expression
\[
\struc{\theta(\x)} = \max_{i = 0, \dots, d} \big(\, \omega_i + \<\x, \al_i\> \,\big),
\]
where $\struc{\omega_i} \in \R$. Let $\struc{\bom} = (\omega_0, \dots, \omega_d)$.
The corresponding \struc{\textit{tropical variety}} is the locus of all points $\x \in \R^n$ such that
the maximum is attained at least twice \cite{MS15}. If all $\al_i$ are distinct,
this locus coincides with the corner locus of the graph
of the function $\theta(\x)$. 

Let $\struc{\wp}$ denote the power set operator. 
There are two polyhedral complexes associated to $\theta(\x)$
obtained from the indicator function
$\struc{\iota} \colon \R^n \rightarrow \wp\,(\{0,\dots, d\})$, defined by
\[
\iota(\x) = \big\{\, i \in \{0, \dots, d\} \, \big| \, \theta(\x) = \omega_i + \<\x, \al_i\> \, \big\}
\]
The \struc{\emph{tropical complex $\struc{M}(\theta)$}} has cells for subsets $I \subset \{0,\dots, d\}$:
\[
\struc{\mu_I} = \big\{ \x \in \R^n \, | \, \iota(\x) \subset I \, \big\},
\]
where we discard empty cells. Each $\mu_I$ is a convex polyhedral cone (or convex polytope) in $\R^n$.
The tropical variety of $\theta$ is the codimension one skeleton of
the complex $M(\theta)$. Let
\[
\struc{\lambda_I} =  \conv\big(\{\al_i \, | \, i \in I\}\big).
\]
Then, the regular subdivision $\Lambda$ associate to $\theta$ is given by
\[
\Lambda = \big\{ \lambda_I \, \big| \, \mu_I \neq \varnothing \, \big\}.
\]
There is a duality relation between the complexes $\Lambda$ and $M$
given by
\[
\check{\mu}_I = \lambda_I.
\]

The following lemma is a multivariate version of Proposition~\ref{Thm:UnivariateUniqueTropical}.
We have chosen the
simplest formulation which suffices to prove the subsequent theorem.
 Let $\bc_1$ and $\bc_2$ be two full-dimensional simplicial
circuits separated by an affine hyperplane $H$. We orient $H$ by choosing a normal
vector $\bv$ such that
\[
\< \bv, \al_1\> \geq \< \bv, \al_2\>, \text{ for all } \al_1 \in \bc_1, \text{ and } \al_2 \in \bc_2.
\]

\begin{lemma}
\label{lem:MultivariateOrdered}
Let $\bc_1$ and $\bc_2$ be full-dimensional simplicial circuits 
whose Newton polytopes intersect
in a common facet $F$, and let $H = \Aff(F)$, with normal vector $\bv$
chosen as above.
For $i= 1, 2$, let $g_i \in \R^{\bc_i}$ be a simplicial agiform with singular loci $\bw_i$.
Let $g(\x) = g_1(\x) + g_2(\x)$.
Then, $g$ belongs to the boundary of the sonc-cone in $\R^{\bc_1 \cup \bc_2}$
if and only if  $\<\bw_1, \bv\> \geq \<\bw_2, \bv\>$.
\end{lemma}

\begin{proof}
By Theorem~\ref{thm:OverlappingConstraint},
either $g$ belongs to the interior of the sonc-cone,
 or the singular loci $\bw_1$ and $\bw_2$ of $g_1$ respectively $g_2$ are contained 
in the same line $\ell$ with directional vector $\bv$.
After a translation, we can assume that $\ell$ passes through the origin,
and define $w_i$ by that $\bw_i  = w_i \bv$.
Consider a translate $\hat \ell$ of $\ell$ passing through the interior
of $\cN(\bc_1)$ and $\cN(\bc_2)$. Let
$\widehat \bc_i$ be a one-dimensional simplicial circuit
contained in $\cN(\bc_i)\cap \hat \ell$.
By Proposition~\ref{pro:DiscSoncSuppWhenAddingPoints},
$g_i$ admits a sonc-decomposition which contains a simplicial
agiform $\hat g_i$ supported on $\widehat \bc_i$.
Let $\hat g (\x) = \hat g_1(\x) + \hat g_2(\x)$.
It suffices to show that either $\hat g$ belongs to the interior of the sonc-cone in 
$\R^{\hat \bc_1 \cup \hat \bc_2}$,
or $\<\bw_1, \bv\> \geq \<\bw_2, \bv\>$.
This follows from Proposition~\ref{Thm:UnivariateUniqueTropical};
the singular loci of univariate exponential sum $x \mapsto \hat g_i( x \bv)$ is 
$w_i$, and $\<\bw_1, \bv\> \geq \<\bw_2, \bv\>$ is equivalent
to that $w_1 \geq w_2$.
\end{proof}

Let us return to the third case in the proof of Theorem~\ref{Thm:UnivariateUniqueTropical},
when we considered two one-dimensional simplicial circuits whose Newton polytopes
where separated by a line segment. In order to complete the proof, we picked a point
in the Newton polytope $\cN$, but outside of the Newton polytopes
of each simplicial circuit. We use a similar trick in the multivariate setting.

\begin{definition}
Let $f \in \R^A$ be an exponential sum with singular sonc-support $S(f)$.
Let $P$ be a polytope with vertices in $A$. We say that $P$ is \struc{\emph{compatible}}
with $S(f)$ if, first, $A$ contains a relative interior point of $P$ and, second,
all simplicial circuits contained in $A\cap P$ are contained in $S(f)$.
Let \struc{$\Gamma_f$} be the set whose
elements are the polytopes $P$ (and all of their faces) that are compatible with $S(f)$
and maximal with respect to inclusion. Denote by $\overline \Gamma_f$ the union of all polytopes $\gamma \in \Gamma_f$.
\end{definition}

\begin{example}
Consider the univariate support set $A = \{0,1, 2, 3, 4, 5, 6, 7\}$, and the
singular support set $S$ generated by the simplicial circuits $c_1 = \{0,1,2\}$,
$c_5 = \{4, 5, 6\}$, and $c_6 = \{5, 6, 7\}$. There are four polytopes $P$
that are compatible with $S$: the line segments $[0, 2]$, $[4, 6]$, $[5, 7]$, and $[4, 7]$. Out of these,
$[0, 2]$ and $[4, 7]$ are maximal with respect to inclusion.
Hence,
\[
\Gamma_f = \big\{ [0,2], [4, 7], 0, 2, 4, 7\big\}.\qedhere
\]
\end{example}

\begin{lemma}
If $\gamma_1, \gamma_2 \in \Gamma_f$ are full-dimensional and distinct, then their intersection
is not full-dimensional.
\end{lemma}

\begin{proof}
Assume that $U = \gamma_1 \cap \gamma_2$ is full-dimensional; we must show that $\gamma_1 = \gamma_2$.
By Proposition~\ref{pro:DiscSoncSuppWhenAddingPoints}, there is no loss of generality 
in assuming that $U$ contains a full-dimensional simplicial circuit $c$.
For $i = 1, 2$, since $\gamma_i \in \Gamma_f$, Proposition~\ref{pro:AllCircuitsSingularPoint} implies that any truncation of a sonc-decomposition
of $f$ to $\gamma_i$ has a singular point $\bw_i$. Since $c$ is contained
in both $\gamma_1$ and $\gamma_2$, we conclude from Theorem~\ref{thm:OverlappingConstraint}
that $\bw_1 = \bw_2$; denote this point by $\bw$.

Assume that there exists $\al \in \gamma_2 \setminus \gamma_1$. 
By Lemma~\ref{prop:BasisSimplexCircuits}, there is a simplicial
circuit $c_0 \subset c \cup \{\al\}$ containing $\al$ as a vertex. Then, $c_0 \subset \gamma_2$,
so that $c_0 \in S(f)$. Hence, there is a sonc-decomposition of $f$ containing
an agiform $g_0$ supported on $c_0$ with $g_0(\bw) = 0$.

Let $A_1 = A \cap \gamma_1$, and consider $A_0 = A_1 \cup \{\al\}$. Since $\gamma_1$
is a maximal element of $\Gamma$, we have that $A_1$
has a relative interior point (i.e., a point of $A_1$ is relative interior to $\cN(A_1)$).
Since $A_1$ is full dimensional, we have that $\dim(A_1) = \dim(A_0)$, 
so that the dimension of the kernel of
$A_0$ is one greater than the dimension of the kernel of $A_1$. If $\bc_1, \dots, \bc_m$
is a basis for $\ker(A_1)$, then $\bc_0$ is linearly independent from
$\bc_1, \dots, \bc_m$, since its $\al$-coordinate is non-zero. Hence,
$\bc_0, \bc_1, \dots, \bc_m$ forms a basis for $\ker(A_0)$.

Since, for $i=0, \dots, m$, each simplicial agiform supported on $c_i$ which occurs in a
sonc-decomposition vanishes at $\bw$,
there is a sonc-decomposition of $f$ that contains the sum
\[
\sum_{i=0}^{m} t_i g_i(\x) = \big\<\varphi_A(\x - \bw), t_0\bc_0 + \dots + t_{m}\bc_{m}\big\>.
\]
where $t_1, \dots, t_, > 0$. It follows that for any simplicial circuit $c$ such that $\bc$
is a linear combination of $\bc_0, \dots, \bc_{m}$, we have that $c \in R(f)$ (cf. proof of Proposition~\ref{pro:DiscSoncSuppWhenAddingPoints}). But
$\bc_0, \dots, \bc_m$ is a basis for $\ker(A_0)$ and, hence, $C_0 \subset R(f)$,
where $C_0$ denote the set of all simplicial circuits in $A_0$.
Since $A_1$ has a relative interior point, and since $S(f)$ is non-empty, we can conclude that
$C_0 \subset S(f)$. It follows that $P = \conv(A_0)$ is compatible with $S(f)$, contradicting
the maximality of $\gamma_1 = \conv(A_1)$. 

We conclude that $\gamma_2 \subset \gamma_1$. Hence, $\gamma_2 = \gamma_1$
by maximality of $\gamma_2$.
\end{proof}

We say that $\Gamma_f$ \struc{\emph{covers}} a set $B$ if $B \subset \overline\Gamma_f$.
We say that $\Gamma_f$ is \struc{\emph{connected}} if $\overline{\Gamma}_f$ is connected.

\begin{thm}
\label{thm:MultivariateTropicalTheorem}
Let $f\in \cS_A^+$ be such that $\Gamma_f$ is connected.
Then, there is a tropical complex $M$,
with dual regular subdivision $\Lambda$, such that
\begin{enumerate}[label=\emph{(}\alph*\emph{)}]
\item $\Lambda$ covers $S(f)$. That is, for each $c \in S(f)$, there exists $\lambda \in \Lambda$ such that $c \subset \lambda$.
\item \label{item:c}If $c \subset\lambda = \check\mu$,
 and $g$ is a simplicial agiform supported on $c$ that appears in a sonc-decomposition  of $f$,
 then $\sing(g) \supset \Aff(\mu)$. 
 \end{enumerate}
\end{thm}

Our proof is similar to that of Theorem~\ref{Thm:UnivariateUniqueTropical}.
In the multivariate setting, there is no longer a finite number of cases, and there is no canonical sonc-decomposition to consider. For readers that are unfamiliar with \emph{regular} subdivision,
we note that a subdivision of the Newton polytope $\cN$ is regular if and
only if there exists a tropical hypersurface whose tropical complex is dual to $\Lambda$.
That is, that $\Lambda$ is regular is implied by the existence of $M$.

\begin{proof}[Proof of Theorem~\ref{thm:MultivariateTropicalTheorem}]
Assume first that $\Gamma_f$ covers the Newton polytope $\cN$.
Then, $\Lambda = \Gamma_f$ is a subdivision of $\cN$ 
(which, a priori, need not be regular).
By definition of $\Gamma_f$, each full-dimensional $\lambda \in \Lambda = \Gamma_f$ contains 
a relative interior point. Using Proposition \ref{pro:AllCircuitsSingularPoint}, we can conclude 
that if $A_\lambda = \lambda \cap A$,
then the truncation of $f$ to $A_\lambda$ has a 
singular point $\bw_\lambda$.
In particular, by Lemma~\ref{lem:TruncationsSingularLocus}, each simplicial circuit contained in $\lambda$
has a singular loci containing $\bw_\lambda$.
Finally, by Lemma~\ref{lem:MultivariateOrdered}, the collection $\{\bw_\lambda\}$ is the vertex set of a  tropical hypersurface dual to $\Lambda$. It follows from the fact that there exists a tropical
hypersurface dual to $\Lambda$, that $\Lambda$ is a regular subdivision. 

\smallskip
For the general case, we give a proof by  minimal counterexample over the volume 
$\vol(\cN \setminus \Gamma_f)$. 
However, as $A$ is a real configuration,
the set of possible volumes is not well-ordered, unless we impose some restrictions.

Assume that there exists a counterexample $f$. That is, assume that $f$ belongs to the boundary of the 
sonc-cone, but there is no tropical complex $M$ fulfilling (\emph{a}) and (\emph{b}).
We can assume that $A$ is minimal in the sense that for any exponential sum
supported on a strict subset of $A$,
there is a tropical complex $M$ fulfilling (\emph{a}) and (\emph{b}).
Then, minimality of $A$ implies that $\Gamma_f$ covers $A$,
for otherwise we could replace $f$ by its truncation to the singular
part of the sonc-decomposition, which would have strictly smaller support.

Consider the family $\cF$ consisting of all exponential sums $h$ such that
\begin{enumerate}[label={(P{\arabic*})}]
\item $A \subset \supp(h) \subset \cN$, \label{item:6Proof1}
\item $h$ belongs to the boundary of the sonc-cone in $\R^{\supp(h)}$, \label{item:6Proof2}
\item $\Gamma_h$ is connected and covers $\supp(h)$, \label{item:6Proof4}
\item any point in $\supp(h) \setminus A$ is an interior point of $\Gamma_h$,
and\label{item:6Proof5}
\item there is no tropical complex $M$ fulfilling (\emph{a}) and (\emph{b}) for $h$.
\end{enumerate}
Then, $f \in \cF$, so that $\cF$ is nonempty.
For each $h \in \cF$, denote by $\Delta_h$ the complement
\[
\struc{\Delta_h} = \cN \setminus \Gamma_h.
\]
By property \ref{item:6Proof5}, $\overline{\Gamma}_h$ is a union 
of polytopes whose vertices are contained in $A$ and, hence, 
there are only finitely many possible values for the
Euclidean volume $v(\Delta_h)$. 
Hence, there exists some $h \in \cF$
such that the volume $v(\Delta_h)$ is minimal.
By the first paragraph of this proof, $v(\Delta_h)$ is positive.

Let $\al \in \Delta_h$ be generic in the sense the each simplicial circuit in
$D = \supp(h) \cup \{\al\}$ containing $\al$ is full-dimensional.
Choose a sequence $\{\varepsilon_j\}_{j \in \N}$ of positive numbers with limit zero.
As in \eqref{eq:ApproximatingSequence}, we obtain, for each $\varepsilon_j$,
an exponential sum
\begin{equation}
\label{eq:ApproximatingSequence2}
h_j(\x) = h(\x)  - \varepsilon_j \,s_1(\varepsilon_j) \, e^{\<\x, \al\>} + \varepsilon_j\, s_2(\varepsilon_j)  \,g_j(\x),
\end{equation}
where $s_1(\varepsilon_j) + s_2(\varepsilon_j) = 1$, which belongs to the boundary of the sonc-cone in  $\R^{D}$. Note that \eqref{eq:ApproximatingSequence2} is not a sonc-decomposition of $h_j(\x)$.

Notice that the monomial $e^{\<\x, \al\>}$ appears in $h_j$ with a negative coefficient.
Hence, for each $j$, there is a simplicial circuit $c_{j} \subset R(h_j)$ with $\al$ as its relative interior point. By choosing a subsequence, we can assume that $c = c_{j}$ is independent of $j$. Fix a sonc-decomposition of $h_j(\x)$ containing a term
\[
  t_j \, \< \varphi_A(\x - \bw_j), \, \bc\>,
\]
where $t_j > 0$.
By Lemma~\ref{lem:LocationOfSingularStrata}, 
the exponential sum $h_{j}$ still belongs to the boundary of the sonc-cone 
if we replace the coefficient $t_j$ with an arbitrary $s_j > 0$.

We now claim that the sequence $\{\bw_j\}_{j \in \N}$ converges in $\R^n$.
If not, then after compactifying $\R^n$ to the polytope $\cN$ using the moment map associated to $A$, it converges to some strict face $F \preccurlyeq \cN$.
Then, with $s_j = \min(\varphi_A(\bw_j))$,
\[
 \lim_{j\rightarrow 0} s_j \, \< \varphi_A(\x - \bw_j), \, \bc\>
\]
is a sum of monomials from $F$ with positive coefficients.
By Lemma~\ref{lem:LocationOfSingularStrata}, we find that
\[
p(\x)= \lim_{j\rightarrow 0} \big(h_{j}(\x) +  s_j \, \< \varphi_A(\x - \bw_j), \, \bc\>\big) 
= h(\x) +  \lim_{j\rightarrow 0}s_j \, \< \varphi_A(\x - \bw_j), \, \bc\>
\]
belongs to the boundary of the sonc-cone.
Hence, a sum including $h$ and monomials from $A$ belongs to the boundary of the sonc-cone.
But $\Gamma_h$ is connected by \ref{item:6Proof4} and, hence, if \emph{any} monomial
appears in a sonc-decomposition of $p$, then repeated use of 
Lemma~\ref{lem:CircuitSoncDecomposition} gives that all monomials appears in some
sonc-decomposition of $p$. Hence, $p$ belongs to the interior of $\Splus_A$ by
Proposition~\ref{pro:InteriorPositiveCombination}, a contradiction.
Hence,  $\{\bw_j\}_{j \in \N}$ converges in $\R^n$. 

By Lemma~\ref{lem:LocationOfSingularStrata}, we find that
\[
p(\x) = \lim_{j\rightarrow 0} \big(h_{j}(\x) +  \, \< \varphi_A(\x - \bw_j), \, \bc\>\big)
= h(\x) +   \lim_{j\rightarrow 0}  \, \< \varphi_A(\x - \bw_j), \, \bc\>
\]
belongs to the boundary of the sonc-cone. That is, $p$ fulfills \ref{item:6Proof2}.
We have that $\supp(p) = \supp(h) \cup \{\al\}$, from which it follows that
$p$ fulfills \ref{item:6Proof1}. 
Since $h$ fulfills  \ref{item:6Proof4}, and $h$ is a summand of $p$, we have that
$M(p)$ is empty. Indeed, as noted above, if a sonc-decomposition of $p$
contains any monomial with a positive coefficient, then there is a sonc-decomposition
of $p$ that contains all monomials, contradiction that $p$ belongs to the boundary of the sonc-cone. 
It follows that the $\Gamma_p$ covers $\Gamma_h$,
from which we conclude, first, that $\Gamma_p$ covers $\supp(p) = \supp(h) \cup \{\al\}$ and, 
second, that $\Gamma_p$ is connected, since $c$ intersects $\supp(h)$.
That is, $p$ fulfills \ref{item:6Proof4}.
To conclude that $p$ fulfills \ref{item:6Proof5}, it suffices to note that
$\al$ is an interior point of $\Gamma_p$, since $c$ is full-dimensional.

We have that $c \in S(p)$. Since $c$ is full-dimensional by construction, and 
intersects the relative interior of $\Delta_h$, we have that $v(\Delta_p) < v(\Delta_h)$.
Hence, $p \notin \cF$, by minimality of $h$.
It follows that there is a tropical complex $M$ fulfilling (\emph{a}) and (\emph{b}) for the exponential sum $p$.
As $h$ is a summand of $p$, the same tropical complex $M$ fulfills 
(\emph{a}) and (\emph{b}) for $h$, contradicting that $h \in \cF$.
\end{proof}

\section{The $\Lambda$-Discriminant}
\label{sec:PositiveDiscriminant}

\subsection{The proof of Theorem~\ref{thm:Main}}
Let us recall the definition of the $\Lambda$-discriminant from the introduction.
Let $A$ be a real configuration, and let $\Lambda$ be a regular subdivision of the Newton polytope $\cN$.
Then, the $\Lambda$-discriminant $D_\Lambda$ 
is the locus of all exponential sums of the form
\[
f(\x) = \sum_{\lambda \in \Lambda} \sum_{c \in C_\lambda} t_{c} \, \<\varphi_A(\x - \bw), \bc\>
 + \sum_{\al \notin A_\Lambda} t_{\al} \,e^{\<\x, \al\>},
\]
where $\bw = \{\bw_\lambda\}_{\lambda \in \Lambda}$ is a family
real parameters such that, for any pair $\lambda_i, \lambda_j \in \Lambda$
whose intersection $\lambda_i \cap \lambda_j$ is nonempty,
\[
\bw_{\lambda_i} - \bw_{\lambda_j} \in \Aff(\lambda_{i} \cap \lambda_{j})^\perp,
\]
and $\bt = \{t_c\}_{c \in C} \cup \{t_{\al}\}_{\al \in A}$ is a family of real parameters.

\begin{proof}[Proof of Theorem~\ref{thm:Main}]
Assume that $f \in \R^A$ belongs to the boundary of the sonc-cone.
If $f$ belongs to a coordinate hyperplane, then there is nothing
to prove. If not, then $S(f) \neq \varnothing$. Hence, $\Gamma_f$
is non-trivial. If $\Gamma_f$ is not connected, then we replace $f$ by its truncation to
one connected component of $\Gamma_f$.
Then, by Theorem~\ref{thm:MultivariateTropicalTheorem},
there is a tropical complex $M$, with dual regular subdivision $\Lambda$,
fulfilling properties (\emph{a}) and (\emph{b}).
We claim that $f$ belongs to $D_\Lambda$.
Indeed, the equations \eqref{eqn:Binomials} describe the linear relations
between vertices of the tropical complex $M$, which finishes the proof.
\end{proof}

\begin{remark}
\label{rem:HornKapranovParameterSpace}
Label the top dimensional cells of $\Lambda$ by $\{\lambda_1, \dots, \lambda_k\}$.
Equation \eqref{eqn:LambdaDiscriminant} gives a parametric
representation of the $\Lambda$-discriminant by a Horn-Kapranov-type map
$\Phi_\Lambda$, which in coordinates is given by
\begin{equation}
\label{eq:HornKapranov}
\Phi_\Lambda(\bt, \bw) =  \sum_{\lambda \in \Lambda} \sum_{c \in  C_\lambda} t_{c} \, \varphi_A(- \bw_{\lambda})* \bc
 + \sum_{\al \notin A_\Lambda} t_{\al} \,e_{\al},
\end{equation}
where $\struc{e_{\al}}$ denote the standard basis vector for the $\al$-coordinate.
Then, the parameter space of $\Phi_\Lambda$ is the
subspace of the product 
$\R^C \times \prod_{j=1}^k \R^n$,
where the coordinates of the $j$-th factor is $\bw_j$,
defined by the equalities \eqref{eqn:Binomials}.
\end{remark}

\subsection{Stratifying the boundary of the sonc-cone.}

Let us consider the integral case, when $A \subset \Z^n$.
Then, the map \eqref{eq:HornKapranov} is algebraic in the parameters
$\bt$ and $e^{\bw}$. In particular, $D_\Lambda$ is a
real algebraic variety.

\begin{remark}
Let $A \subset \Z^n$.
If $\Gamma_f$ is not connected, then each connected component of
$\Gamma_f$ gives a regular subdivision $\Lambda$ such 
that $f \in D_\Lambda$, where the implicit representations of
$D_\Lambda$ are in complementary sets of variables.
In particular, if we only wish to describe
the codimension one pieces of the boundary of the sonc-cone,
then there is no loss of generality in assuming that $\Gamma_f$ is connected.
\end{remark}

\begin{definition}
\label{def:Strata}
Let $A$ be integral.
Let $\struc{S_\Lambda} \subset D_\Lambda$ denote the subset of the $\Lambda$-discriminant
determined by that $\bt \geq 0$ and that $\bw$ arranged according to a tropical hypersurface dual to $\Lambda$. We call $S_\Lambda$ the \struc{\emph{strata}} associated to $\Lambda$.
\end{definition}

The benefit of considering $S_\Lambda$ rather than $D_\Lambda$, is that $S_\Lambda$
is completely contained in the sonc-cone.
Our main Theorem~\ref{thm:Main} can be sharpened to the following theorem.
\begin{thm}
\label{thm:MainStrata}
Let $A$ be integral. Then, the boundary of the sonc-cone $\Splus_A$ is contained in the union of the coordinate hyperplanes and
the strata $S_\Lambda$,
as $\Lambda$ ranges over all regular subdivision of the Newton polytope $\cN(A)$.
\end{thm}

\begin{proof}
Since the locus of exponential sums $f$ which has that $\Gamma_f$ disconnected 
is smaller dimensional, there is no loss
of generality in assuming that $\Gamma_f$ is connected.
Then, we can apply the same proof as for Theorem~\ref{thm:Main}.
\end{proof}

Let $A$ be integral. By Theorem~\ref{thm:MainStrata}, we have a covering of 
the boundary of the sonc-cone by semi-algebraic sets. However, in general, this covering need
not be a stratification.

\begin{remark}
\label{rem:StructureTheoremDimensionOne}
Let $A$ be a univariate integral configuration.
Each regular subdivision $\Lambda$ defines a unique
singular sonc-support, determined by an index set $I$
as in \eqref{eqn:UnivariateSimplicialSoncSupport}.
Let us denote by $S_I$ the set of all exponential 
sums $f(x)$ that has a sonc-decomposition as in \eqref{eqn:UnivariateBoundarySoncDecomposition},
where the singular loci $w_{i_1}, \dots, w_{i_k}$ fulfill that
\begin{equation}
\label{eqn:OrderedRoots}
w_{i_1} \leq \dots \leq w_{i_k}.
\end{equation}
We can stratify $S_I$ further, according to which of the inequalities 
\eqref{eqn:OrderedRoots} that are strict and which
are equalities. We label such a strata by grouping the indices
for which the singular loci are equal, with groups separated by a vertical bar. 
By Theorem~\ref{thm:OverlappingConstraint}, any two subsequent indices which
belong to $I$ must be grouped together.
\end{remark}

\begin{example}
\label{ex:01234}
Consider the family of univariate quartic polynomials
\[
f(z) = 1 + a_1 z + a_2 z^2 + a_3 z^3 +  z^4.
\]
associated to the support set $A = \{0,1,2,3,4\}$.
There are two two-dimensional strata
in the affine space $a_0 = a_4 = 1$,
labeled by $\{1,2,3\}$ and $\{1 \, | \, 3\}$.
\begin{table}[t]
\begin{center}
\def\arraystretch{1.5}
\begin{tabular}{c|l}
dim & S \\
\hline
2 & $\{1,2,3\}, \{1 \,|\, 3\}$ \\
1 & $\{1, 3\}$ \\
\end{tabular}
\end{center}
\caption{Dimension and labeling of the strata of the boundary of $\Splus_A$ for univariate quartics from Example \ref{ex:01234} in the affine piece $a_0 = a_d = 1$.}
\label{fig:01234b}
\end{table}
Let the notation of circuits $c_i$ be as in \eqref{Equation:DistinguishedUnivariateCircuits}.

Consider first the strata $S_{\{1,2,3\}}$. The index should be read as that 
all three simplicial circuits $\bc_1, \bc_2,$ and $\bc_3$ appears in one
group. That is, the simplicial agiforms supported on these circuits should
have a common singular point. Hence, $S_{\{1,2,3\}}$ consists of 
all polynomials of the form
\begin{align}
	f(z) = \<\varphi_A(z w^{-1}), t_1\bc_1+t_2\bc_2 + t_3 \bc_3\>. \label{Equation:ExampleUnivariateQuartics1}
\end{align}
By construction, such a polynomial $f(z)$ has a singular point at $z = w$, since
\[
f(w) = \<\varphi_A(1),  t_1\bc_1+t_2\bc_2 + t_3 \bc_3\> = 0
\]
and, therefore, the stratum $S_{\{1,2,3\}}$ is contained in the boundary of $\Pplus_A$.
Indeed, $S_{\{1,2,3\}}$ is contained in the classical $A$-discriminant; see the left picture of Figure \ref{fig:01234example}. 

Consider now the strata $S_{\{1\, | \,3 \}}$. The index should be read as
that the simplicial circuits $\bc_1$ and $\bc_3$ appear in separate groups.
That is, the stratum
$S_{\{1\, | \,3 \}}$ consists of all polynomials of the form
\begin{align}
	f(z)  = \<\varphi_A (zw_1^{-1}), t_1 \bc_1\> + \<\varphi_A(z w_3^{-1}), t_3 \bc_3\>.  \label{Equation:ExampleUnivariateQuartics2}
\end{align}
By Theorem \ref{Thm:UnivariateUniqueTropical} we have $w_1 \leq w_3$ and, in general, 
these two points do not need to coincide.
If $w_1 < w_3$, then at every point in $\R_{>0}$ at least one of the two exponential sums is
positive. Therefore, this strata comprises a part of the boundary, which is
contained in the strict interior of the nonnegativity cone $\Pplus_A$.

Finally, let us investigate the intersection $S_{\{1,3\}} = S_{\{1,2,3\}} \cap S_{\{1 \,|\, 3\}}$.
Considering an exponential sum $f \in S_{\{1 \,|\, 3\}}$ of the form \eqref{Equation:ExampleUnivariateQuartics2}. Then, $f \in S_{\{1,3\}}$ if and only $w_1 = w_3$.
That is, $S_{\{1,3\}}$ consists of all polynomials of the form
\begin{align}
f(z)  = \<\varphi_A (zw^{-1}), t_1 \bc_1 + t_3 \bc_3\>.
\end{align}

The algebraic hypersurfaces defined by the $\Lambda$-discriminants $D_{\{1,2,3\}}$ and $D_{\{1\, | \, 3\}}$
are pictured in Figure~\ref{fig:01234example},
together with the boundary of the sonc-cone.
\end{example}
\begin{figure}[t]
\includegraphics[width=45mm]{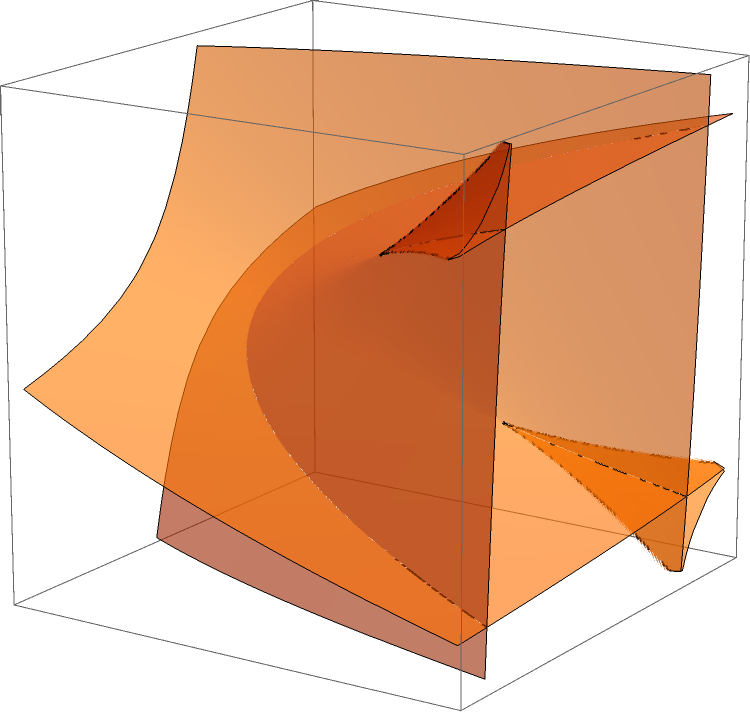}
\hspace{7mm}
\includegraphics[width=45mm]{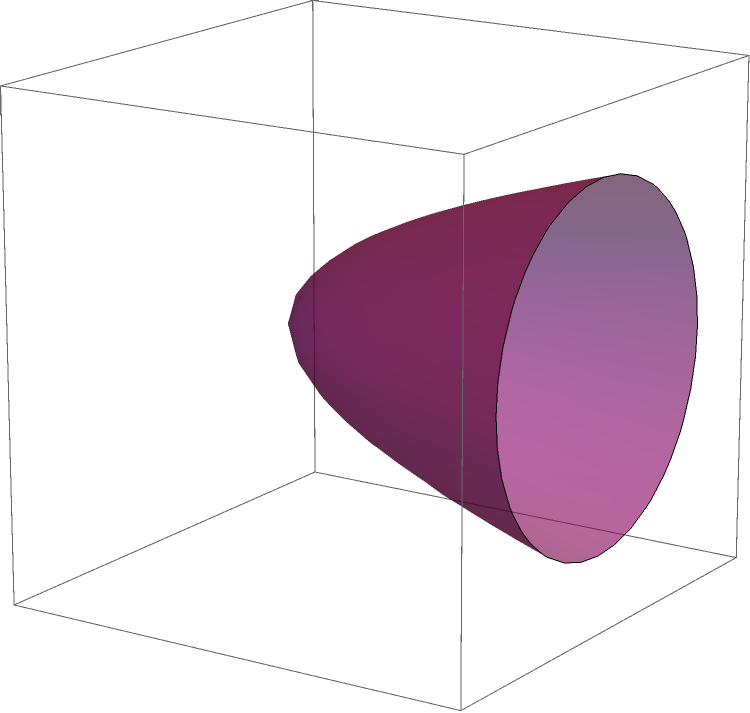}
\hspace{7mm}
\includegraphics[width=45mm]{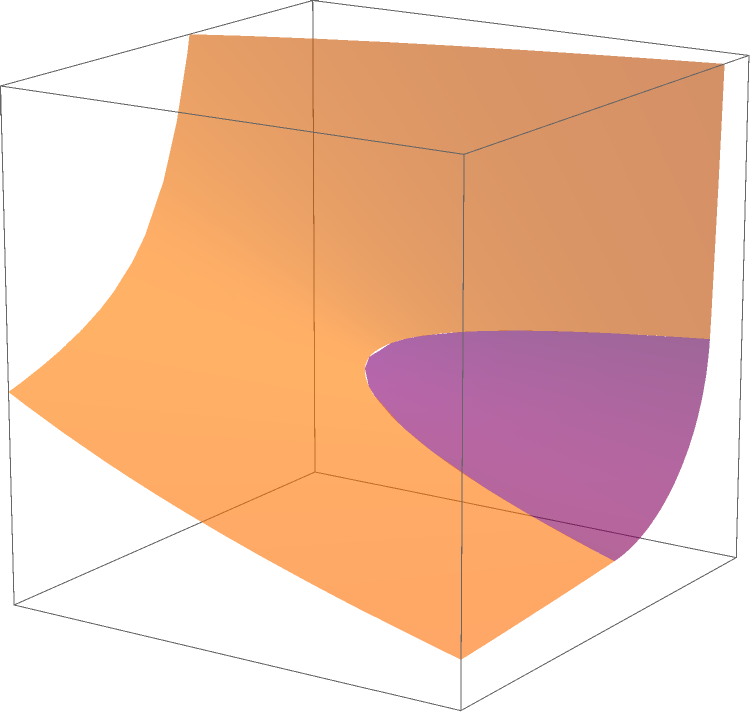}
\caption{The real Zariski-closures of the strata $S_{\{1,2,3\}}$ and $S_{\{1 \,|\, 3\}}$, 
and the boundary of the sonc-cone of univariate quartics from Example \ref{ex:01234}.
}
\label{fig:01234example}
\end{figure}

\begin{corollary}
The closure of $S_I$ is the union of all $S_J$, where $J$ 
can be obtained from $I$ by deleting indices or deleting bars. 
The dimension of $S_I \subset \R^A$, is the sum of the cardinality of $I$ and the number of groups.
\qed
\end{corollary}

In the univariate case, the regular subdivision $\Lambda$ uniquely determines
the sonc-support of an exponential sum $f \in S_\Lambda$.
It follows that two strata $S_\Lambda$
only intersect along common lower dimensional strata.
The main issue, in the multivariate case, is that 
two positive $\Lambda$-discriminants which do not coincide, may overlap
in a common Zariski-dense subset. The reason for this is that two regular
subdivisions can give rise to the same discriminants,
if they differ only in cells which do not contain any relative interior points
(e.g., compare two unimodular triangulations).
However, cells of the subdivision $\Lambda$ which contain no relative interior points, 
can still impose non-trivial restrictions on $\bw$ in Definition \ref{def:Strata}.
That is, two regular subdivisions which gives rise to the same discriminants,
need not give rise to the same strata.

\subsection{Properties of algebraic $\Lambda$-discriminants}
For completeness, let us mention some important properties of
$\Lambda$-discriminants in the integral case.

\begin{prop}
\label{pro:AtMostHypersurface}
Let $A$ be integral. Then, the (complex Zariski closure of the) $\Lambda$-discriminant is a unirational algebraic variety of codimension at least one.
If the codimension is one, then $\overline{D}_\Lambda$ is rational.
\end{prop}

\begin{proof}
That $D_\Lambda$ is a unirational algebraic variety in the integral case
follows from the fact that the equations \eqref{eqn:Binomials} 
translates to a set of binomial equations in the variable $\bz = e^{\x}$,
and hence the map $\Phi_\Lambda$ is defined on a toric (and, hence, rational) variety.

Label the top dimensional cells of $\Lambda$ by $\{\lambda_1, \dots, \lambda_k\}$,
as in Remark~\ref{rem:HornKapranovParameterSpace},
and write $A_i = A \cap \lambda_{i}$.
Let $d$ denote the cardinality of $A$, so that $\R^A$ has dimension $d$.
For any subset $I \subset \{1, \dots, k\}$, let $d_I$, $n_I$, and $m_I$
denote the cardinality, dimension, and codimension of
the intersection
\[
A_I = \bigcap_{i \in I} A_i,
\]
By the inclusion-exclusion principle, we have that
\[
n =  \sum_{j = 1}^k (-1)^{j+1}\sum_{|I| = j} n_I,
\quad \text{ and } \quad
m =  \sum_{j = 1}^k (-1)^{j+1}\sum_{|I| = j} m_I.
\] 
Let us compute the dimension of the parameter space of $\Phi_\Lambda$.
Taking the relations \eqref{eqn:Binomials} into account,
we conclude that $\bw$ has $n$ independent coordinates.
Discarding scaling factors $t$ for the circuits $\bc$ which  belong to several top-dimensional cells $\lambda_j$,
we conclude that $\bt$ has $m$ independent coordinates.
Hence, the codimension of $D_\Lambda$ is at least
$d -  n - m = 1$.
\end{proof}

\begin{corollary}
Let $A$ be integral. Then, for each regular subdivision $\Lambda$ of $\cN(A)$,
the Zariski closure $\overline{D}_\Lambda$ is irreducible. \qed
\end{corollary}

From the semi-explicit representation \eqref{eq:HornKapranov}, 
whose parameters fulfill the relations \eqref{eqn:Binomials},
it is straightforward (but expensive) to compute an implicit representation of $\overline{D}_\Lambda$
using Gr{\"o}bner basis techniques.

\begin{figure}[t]
\begin{tabular}{c|c|c}
\includegraphics[height=15mm]{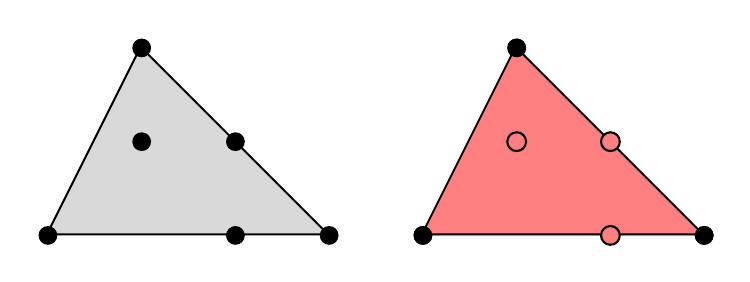}
&
\includegraphics[height=15mm]{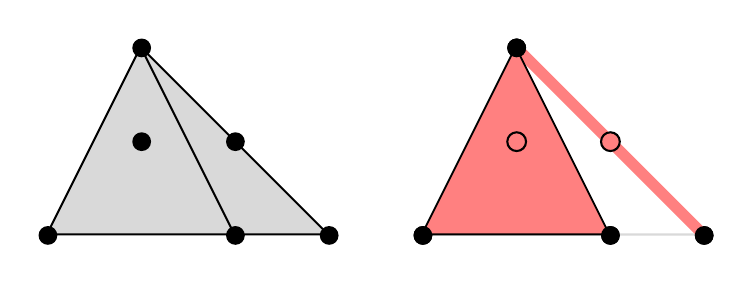}
&
\includegraphics[height=15mm]{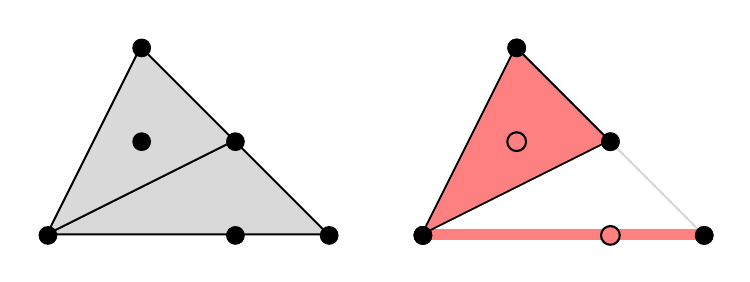}
\\ \hline
\includegraphics[height=15mm]{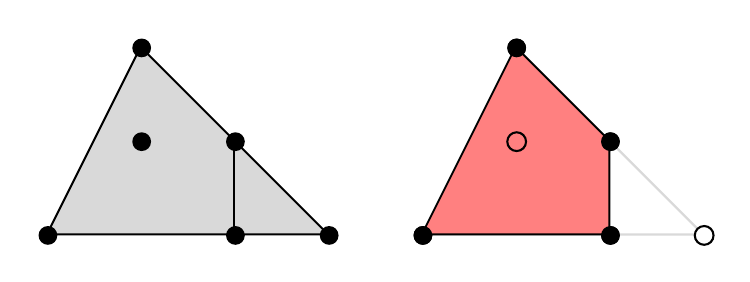}
&
\includegraphics[height=15mm]{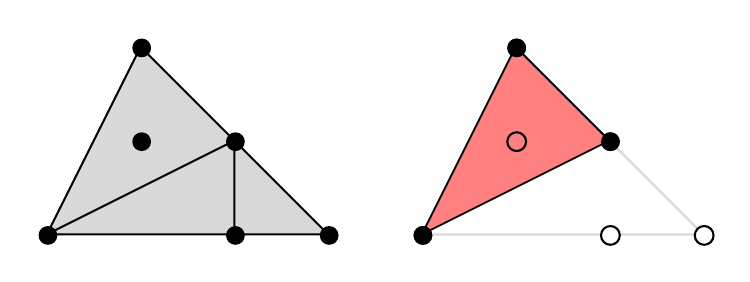}
&
\includegraphics[height=15mm]{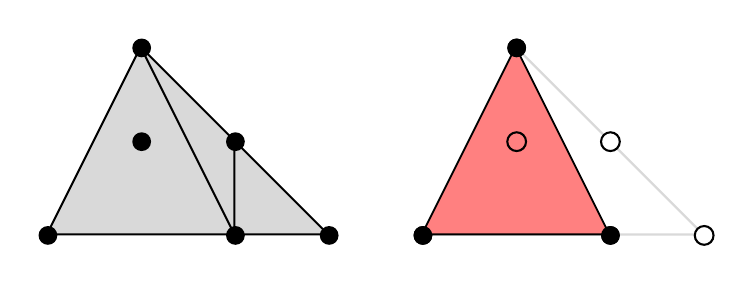}
\\ \hline
\includegraphics[height=15mm]{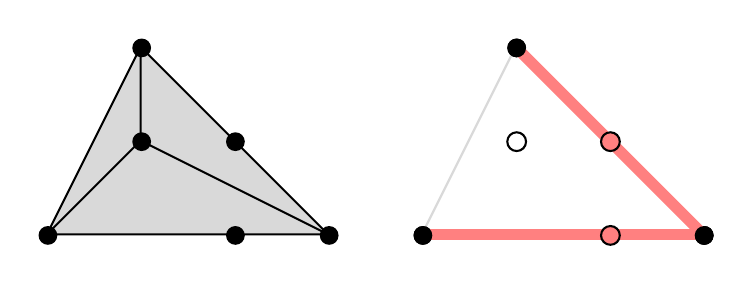}
&
\multicolumn{2}{c}{
\includegraphics[height=15mm]{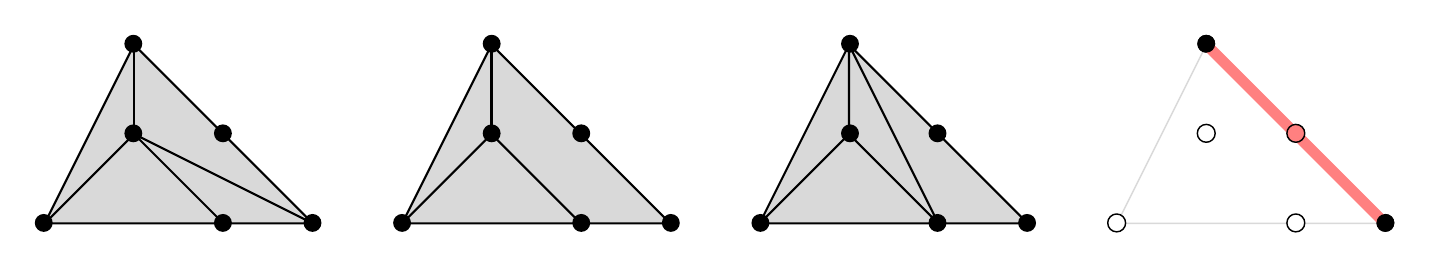}
}
\\ \hline
\includegraphics[height=15mm]{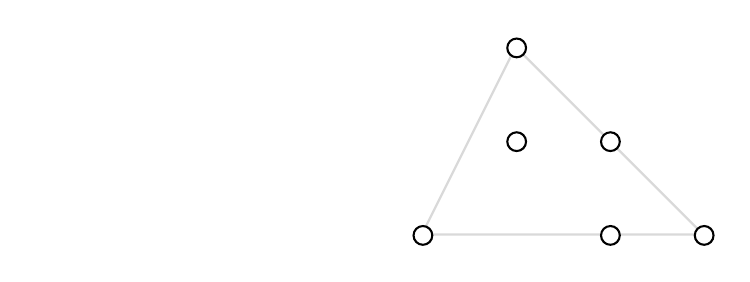}
&
\multicolumn{2}{c}{
\includegraphics[height=15mm]{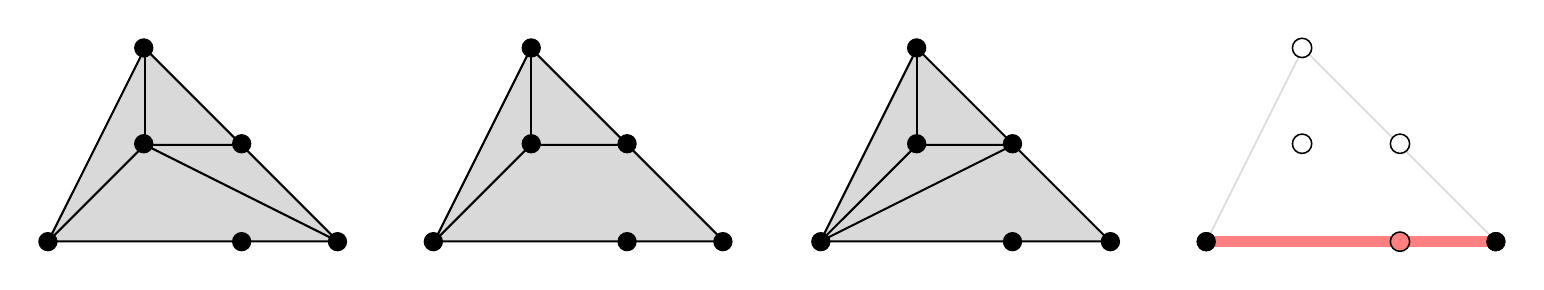}
}
\end{tabular}
\caption{Regular subdivisions of the Newton polytope in
Example~\ref{ex:SimpexCone}, and simplicial circuits
which are contained in some cell of the subdivision.}
\label{fig:SimplexExample}
\end{figure}

\begin{example}
\label{ex:SimpexCone}
Let 
\[
A = \left[ \begin{array}{cccccc} 
1 & 1 & 1 & 1 & 1 & 1 \\ 
0 & 2 & 3 & 1 & 2 & 1\\
0 & 0 & 0 & 1 & 1 & 2
\end{array}\right]
\quad \text{ and } \quad
C^{\top} = 
\left[ \begin{array}{rrrrrr} 
1 & 1 & 0 & -4 & 0 & 2 \\ 
0 & 0 & 1 & 0 & -2 & 1\\
1 & -3 & 2 & 0 & 0 & 0\\
1 & 0 & 0 & -3 & 1 & 1
\end{array}\right],
\]
Three codimension one pieces of the boundary of the sonc-cone
are contained in the coordinate hyperplanes
\[
a_0 = 0, \quad a_2 = 0, \quad \text{ and } \quad a_5 = 0.
\]
In the upper left corner of Figure~\ref{fig:SimplexExample} is the trivial regular subdivision
$\Lambda_0$, which contains $\cN$ as the only full-dimensional cell. 
The $\Lambda$-discriminant $D_0$, which coincides with the $A$-discriminant
$D_A$, has the explicit presentation
\[
\ba = \varphi_A(\bw^{-1})*(t_1 \bc_1 + t_2 \bc_2 + t_3 \bc_3 + t_4 \bc_4).
\]
A Gr{\"o}bner basis computation yields the implicit description
\begin{align*}
D_0(\ba) = & -a_2 a_3^6 + a_1 a_3^5 a_4 + a_0 a_3^3 a_4^3 - a_1^2 a_3^4 a_5 - 
 36 \,a_0 a_2 a_3^3 a_4 a_5 + 30\, a_0 a_1 a_3^2 a_4^2 a_5 + 27 \,a_0^2 a_4^4 a_5 \\
 &+ 72 \,a_0 a_1 a_2 a_3^2 a_5^2 - 96 \,a_0 a_1^2 a_3 a_4 a_5^2 - 
 216\, a_0^2 a_2 a_4^2 a_5^2 + 64 \,a_0 a_1^3 a_5^3 + 432 \,a_0^2 a_2^2 a_5^3.
\end{align*}
The regular subdivision $\Lambda_1$ in the top of the middle column in 
Figure~\ref{fig:SimplexExample} gives the singular sonc-support consisting 
of the two simplicial circuits $\bc_1$ and $\bc_2$.
The regular subdivision $\Lambda_1$ has two top-dimensional cells
$\lambda_1$ and $\lambda_2$, to which we associate two points 
$\bw_1 = (w_{11}, w_{12})$ and $\bw_2 = (w_{21}, w_{22})$.
The cells $\lambda_1$ and $\lambda_2$ intersect in the line 
segment with tangent vector $\bu = (1, -2)$.
That is, the positive discriminant $D_1(\ba)$ admits the representation
\[
\ba = \varphi_A(\bw_1^{-1})*(t_1 \bc_1) +  \varphi_A(\bw_2^{-1})*(t_2 \bc_2), \quad \text{where} \quad
\bw_1^{\bu} = \bw_2^{\bu}.
\]
A Gr{\"o}bner basis computation yields the explicit representation
\[
D_1(\ba) = -a_2^2 a_3^4 + 4 \,a_0 a_1 a_4^4 - 32 \,a_0 a_1 a_2 a_4^2 a_5 + 64\, a_0 a_1 a_2^2 a_5^2.
\]
The computations for the regular subdivision $\Lambda_2$ in the upper right corner of
Figure~\ref{fig:SimplexExample} are similar, and yields the explicit representation
\[
D_2(\ba) =a_2^2 a_3^3 + 4\, a_1^3 a_4 a_5 + 27\, a_0 a_2^2 a_4 a_5.
\]
The regular subdivision on the second row of Figure~\ref{fig:SimplexExample}
do not yield codimension one pieces of the boundary of the sonc-cone,
as the points of $A$ not covered by $\Gamma_f$ cannot appear in a sonc-decomposition by Corollary~\ref{cor:RegularSoncSupporIntersection}.

Consider the regular subdivision $\Lambda_3$ in the left column of the third row of
Figure~\ref{fig:SimplexExample}. Here, $\Gamma_f$ consists
of the two one-dimensional circuits $\bc_2$  and $\bc_3$.
The point $\al_3 = (1,1)$ is not contained in $\Gamma_f$.
In this case, there are no obstructions to adding $\al_3$ to the sonc-support.
Hence, we obtain a codimension one piece of the boundary.
The regular subdivision has three full-dimensional cells, of which only two contains
simplicial circuits. The separating hyperplane has tangent $\bu = (2, -1)$.
\[
\ba = \varphi_A(\bw_2^{-1})*(t_2 \bc_2) +  \varphi_A(\bw_3^{-1})*(t_3 \bc_3) + a_3 \bbe_3 \quad \text{and} \quad
\bw_2^{\bu} = \bw_3^{\bu}.
\]
where $\boldsymbol{e}_i$ denotes the $(i+1)$-st standard basis vector for $i=0, \dots, d$.
A Gr{\"o}bner basis computation
yields the explicit representation
\[
D_3(\ba) = 27 \,a_0 a_4^4 - 216 \,a_0 a_2 a_4^2 a_5 + 64 \,a_1^3 a_5^2 + 432 \, a_0 a_2^2 a_5^2.
\]
The two remaining regular subdivision are induced by faces $F$ of the Newton polytope
$\cN(A)$. The corresponding positive discriminants are simply the positive
$(A\cap F)$-discriminants of the corresponding faces, and they are given by
\[
D_4(\ba) = a_4^2 - 4\, a_2 a_5
\quad \text{and} \quad
D_5(\ba) = 4\,a_1^3 - 9\,a_0 a_1^2.
\]
The boundary of the sonc-cone in the affine space $2a_0 = a_2 = a_5 = 2$
is shown in Figure~\ref{fig:SimplexCone},
with the discriminants $D_0(\ba)$ and $D_1(\ba)$.
In this figure, the boundary of the sonc-cone seems to be disconnected is
an artifact of numerical instabilities, which occur due to the severe
singularities of $D_1(\ba)$ along the line
$\ba(s) = (1, 0, 2, s , -4, 2)$.
\begin{figure}[t]
\includegraphics[width=45mm]{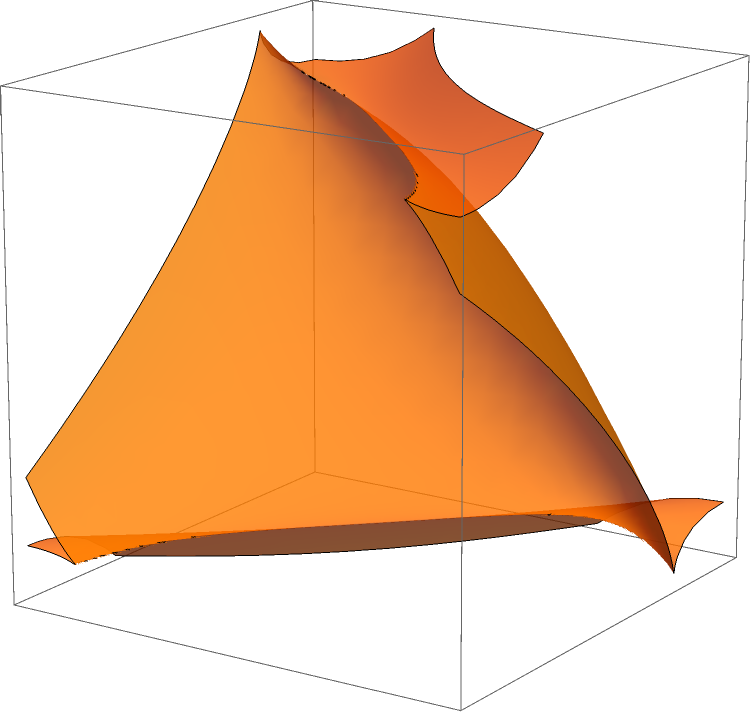}
\hspace{7mm}
\includegraphics[width=45mm]{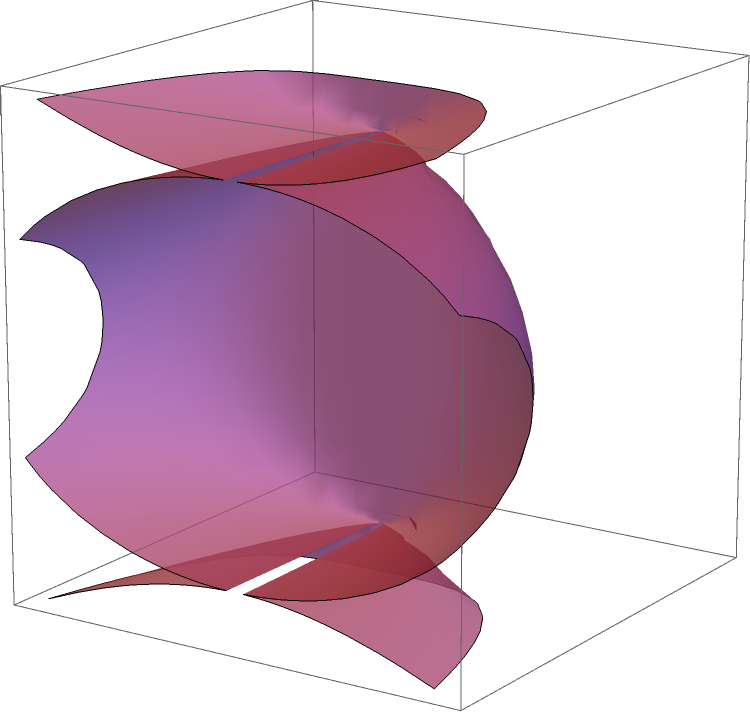}
\hspace{7mm}
\includegraphics[width=45mm]{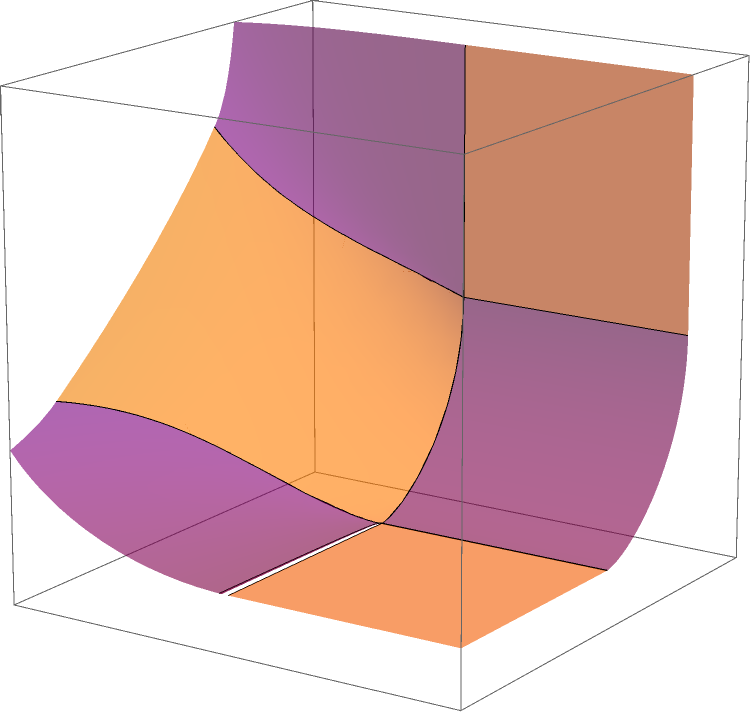}
\caption{The varieties defined by $D_0(\ba)$ and $D_1(\ba)$, 
and the boundary of $\Splus_A$, from Example \ref{ex:SimpexCone}.}
\label{fig:SimplexCone}
\end{figure}
\end{example}

\section{On the equality $\Splus_A = \Pplus_A$}
\label{sec:Equality}

We deduce, as a corollary of our description of the boundary of the sonc-cone,
a combinatorial characterization of when the sonc-cone is equal to the nonnegativity cone.

\begin{thm}
\label{thm:NonnegativeIsSonc}
Let $A$ be a support set, generic in the sense that all simplicial circuits $\bc \in C$
are full-dimensional. Then, $\Splus_A = \Pplus_A$ if and only if for each $f\in \Splus_A$, the set 
$\Gamma_f$ is defined by a unique top-dimensional cell $\gamma$.
\end{thm}

\begin{proof}
It suffices to show that each exponential sum $f$ which belongs to the
boundary of $\Splus_A$ belongs to the boundary of $\Pplus_A$.
If $\Gamma_f$ is empty, then some monomial corresponding to a vertex of $\cN(A)$ has a vanishing coefficient,
implying that $f$ belongs to the boundary of $\Pplus_A$.
If the $\Gamma_f$ is non-empty, then by assumption $\gamma$ is
a sub-polytope of $\cN(A)$ fulfilling that $c \subset \gamma$ implies that $c \in S(f)$. 
Since $S(f)$ is non-empty, and all simplicial circuits are full-dimensional,
Corollary~\ref{cor:RegularSoncSupporIntersection} implies that $M(f) = \varnothing$.
Then, $f$ is effectively supported on $A \cap \gamma$, and $S(f)$ is the set of all simplicial
circuits contained in $A\cap \gamma$.
Hence, by Proposition~\ref{pro:AllCircuitsSingularPoint},
$f$ has a singular point, so that $f$ belongs to the boundary of $\Pplus_A$.
\end{proof}

In \cite[Conjecture 22]{Chandrasekaran:Murray:Wierman}, it was conjectured that
 the sonc-cone is equal to the nonnegativity cone
only if either the support set consists of the vertices of a simplex with two additional non-vertices, or if there is a unique non-vertex. We provide a family of counterexamples to this conjecture. 

\begin{remark}
\label{rem:lastexample}
Let $V$ be the vertices of the polytope $\cN(V)$, and assume that the origin is contained in the interior of $\cN(V)$. Let $F \preccurlyeq \cN(V)$ be a simplicial face of $\cN(V)$. 
Set $V_F = V\cap F$, and consider the
support set 
\[
A = V \cup (1-\varepsilon) V_F,
\]
where $\varepsilon > 0$ is small.
If the origin is in generic position relative to the vertices of $\cN(V)$,
then the support set $A $ contains only full-dimensional
circuits. 

If the origin is sufficiently close to the relative interior of the face $F$, 
then the origin lies in the complement polytope
$\cN(V \setminus \{\al\})$ for each $\al \in V_F$. 
This implies that $A \setminus \{\al, (1-\varepsilon) \al\}$ contains no simplicial circuits.
It follows that each simplicial circuit $c \subset A$ contains $(1-\varepsilon) V_F$.
We conclude that the Newton polytope of any two simplicial circuits 
overlap in a full-dimensional domain.
Hence, Theorem~\ref{thm:NonnegativeIsSonc}
implies that $\Splus_A = \Pplus_A$.
\end{remark}

\begin{example}
\label{ex:Equality1}
If $F$ is the full face, so that $V$ is the vertices of a simplex, then we obtain a configuration $A$ with $n+1$ interior points such that  $\Splus_A = \Pplus_A$.
\end{example}

\begin{example}
\label{ex:lastexample}
To obtain an example of a configuration such that $\cN(A)$ is not a simplex,
which contains more than one interior point, consider the configuration
\[
A = \left[\begin{array}{cccccc}
1 & 1 & 1 &1 & 1 & 1\\
0 & 0 & 1 & 1 & 5 & 5\\
0 & 5 & 2 & 3 & 0 & 5
\end{array}\right],
\]
We have depicted in Figure~\ref{fig:equality1} the five (up to symmetries) possible
sets $\overline{\Gamma}_f$ associated to $A$, each of which consists of a single two-dimensional cell.
Hence, $\Splus_A = \Pplus_A$.
\begin{figure}[t]
\includegraphics[height=25mm]{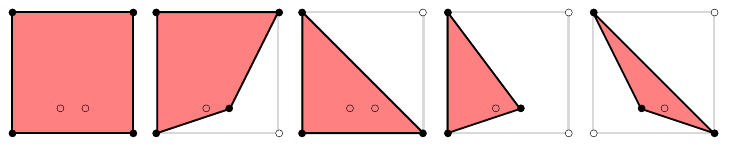}
\caption{The five sets $\overline{\Gamma}_f$ from Example \ref{ex:lastexample}.}
\label{fig:equality1}
\end{figure}
\end{example}

\newcommand{\etalchar}[1]{$^{#1}$}
\providecommand{\href}[2]{#2}
\bibliographystyle{amsalpha}

\end{document}